\documentclass[12pt]{amsart}

\usepackage{amscd,amsrefs}
\usepackage{geometry}
\numberwithin{equation}{section}

\newtheorem{thm}{Theorem}[section]
\newtheorem{theorem}[thm]{Theorem}
\newtheorem{lemma}[thm]{Lemma}
\newtheorem{cor}[thm]{Corollary}
\newtheorem{prop}[thm]{Proposition}

\theoremstyle{definition}

\newtheorem{remark}[thm]{Remark}

\newtheorem{defn}[thm]{Definition}

\newtheorem{defn-thm}[thm]{Definition-Theorem}

\usepackage{graphicx}
\usepackage{mathtools}
\DeclarePairedDelimiter{\ceil}{\lceil}{\rceil}
\DeclarePairedDelimiter{\floor}{\lfloor}{\rfloor}

\usepackage{amsmath,amsfonts,amsthm,amssymb,graphics,color,datetime}

\usepackage{hyperref}
\usepackage{tikz}
\usetikzlibrary{angles,quotes,calc}
\newcommand{\RN}[1]{%
  \textup{\uppercase\expandafter{\romannumeral#1}}%
}

\usepackage{tikz}
\usetikzlibrary{decorations.markings}
\usetikzlibrary{angles,quotes,calc}
\usetikzlibrary{calc}
 \usepackage{subfigure}

\def\beq{\begin{eqnarray}}
\def\eeq{\end{eqnarray}}
\newcommand{\nn}{\nonumber}

\newcommand{\Rea}{\operatorname{Re}} 

\numberwithin{equation}{section}

\newcommand{\cO}{\mathcal{O}}
\newcommand{\Res}{\operatorname{Res}}
\newcommand{\Ima}{\operatorname{Im}}

        \definecolor{pink}{rgb}{1,0,1}
        \definecolor{purple}{rgb}{0.4,0.2,1}

\newcommand{\re}{\operatorname{Re}}

\newcommand{\pa}{\partial}
\newcommand{\eps}{\varepsilon}

\newcommand{\cB}{{\mathcal{B}}}

\newcommand{\N}{\mathbb{N}}
\newcommand{\R}{\mathbb{R}}

\newcommand{\C}{\mathbb{C}}
\newcommand{\Z}{\mathbb{Z}}

\begin{document}

\title{The variation of Barnes and Bessel zeta functions}  

\author[C. L. \@ Aldana]{Clara L. Aldana}
\address{Departamento de Matem\'aticas y Estad\'{\i}stica. Universidad del Norte. Km 5 Via Puerto Colombia.  Area Metropolitana de Barranquilla, Colombia}
\email{claldana@uninorte.edu.co, clara.aldana@posteo.net}

\author[K. \@ Kirsten]{Klaus Kirsten} 
\address{GCAP-CASPER, Department of Mathematics, Baylor University Waco, Texas, 76798, USA and Mathematical Reviews, American Mathematical Society, 416 4th Street, Ann Arbor, MI 48103, USA}
\email{Klaus\_Kirsten@baylor.edu}

\author[J.\@ Rowlett]{Julie Rowlett}
\address{Department of Mathematics, Chalmers University of Technology and the University of Gothenburg, 41296 Gothenburg, Sweden}
\email{julie.rowlett@chalmers.se}

\maketitle

\begin{abstract}
We consider the variation of two fundamental types of zeta functions that arise in the study of both physical and analytical problems in geometric settings involving conical singularities. These are the Barnes zeta functions and the Bessel zeta functions. Although the series used to define them do not converge at zero, using methods of complex analysis we are able to calculate the derivatives of these zeta functions at zero.  These zeta functions depend critically on a certain parameter, and we calculate the variation of these derivatives with respect to the parameter. For integer values of the parameter, we obtain a new expression for the variation of the Barnes zeta function with respect to the parameter in terms of special functions.  For the Bessel zeta functions, we obtain two different expressions for the variation via two independent methods.  Of course, the expressions should be equal, and we verify this by demonstrating several identities for both special and elementary functions.  We encountered these zeta functions while working with determinants of Laplace operators on cones and angular sectors. 
\end{abstract}

\section{Introduction}
Zeta functions are ubiquitous in mathematics and physics.  The first zeta function one typically encounters as a student is Riemann's zeta function,
\[ \zeta_R (z) = \sum_{n \geq 1} n^{-z}. \]
Although it was not defined by Riemann, it is named after him because he demonstrated a relationship between its zeros and the distribution of prime numbers \cite{riemann} that turned out to be essential to the first proofs of the prime number theorem by Hadamard \cite{Hadamard1896} and de la Vall\'ee Poussin \cite{dlvp1896} in 1896.  This zeta function was introduced by Euler in the early 1700s who showed, among other things, that it is equal to a product over prime numbers:
\[ \zeta_R(z) = \prod_{p \, \, \mathrm{ prime }} \frac{1}{1-p^{-z}}. \]
This may hint towards a relationship between Riemann's zeta function and the distribution of prime numbers.  More generally, every global field has a zeta function that is intimately related to the distribution of its primes.  The Riemann zeta function is associated to the field of rational numbers.  This motivates the investigation of zeta functions in number theory.

However, zeta functions are also essential objects in physics and geometric analysis.  To see this connection, recall the one-dimensional Helmholtz equation on a bounded interval $[0,\ell]$ with the Dirichlet boundary condition:
\[ f''(x) + \lambda f(x) = 0, \quad f(0) = f(\ell) = 0. \]
This is equivalently the eigenvalue equation for the Laplace operator; the goal is to find all functions $f$ (not identically zero) and numbers $\lambda$ that satisfy the equation.  Solving the Helmholtz equation is an essential step in solving both the heat and wave equations.  Using calculus one can show that
\[ f_n(x) = \sin \left( \frac{n \pi x}{\ell} \right), \quad \lambda_n = \frac{n^2 \pi^2}{\ell^2}, \quad n \in \N_{+}\]
are solutions.  Harnessing techniques from Fourier analysis \cite{folland} and the general spectral theory of self-adjoint operators, one can show that these functions are (up to multiplication by constants) all the eigenfunctions and therefore the full set of eigenvalues is $\{ \lambda_n \}_{n \geq 1}$. 
Since the collection of eigenvalues of an operator is known as its \em spectrum, \em a \em spectral zeta function \em is defined to be
\[\zeta(z) = \sum_{n \geq 1} \lambda_n ^{-z}. \]
For this particular example,
\[\sum_{n \geq 1} \lambda_n ^{-z} = \sum_{n \geq 1} \left( \frac{n^2 \pi^2}{\ell^2} \right)^{-z} = \frac{\ell^{2z}}{\pi^{2z}} \zeta_R(2z).\]
Spectral zeta functions can be used to define the determinant of unbounded operators, like the Laplace operator, $\Delta$, if one can make sense of
\[ \zeta'(0) \implies \det(\Delta) = e^{-\zeta'(0)}. \]
This is known as the \em zeta-regularized determinant.\em \ This determinant was introduced by Ray and Singer in \cite{Ray-Singer} in connection to analytic torsion. Later,  Stephen Hawking showed that it can be used to obtain an energy momentum tensor which is finite even on the horizon of a
black hole \cite{hawking77}. The zeta regularized determinant gives a valuable functional on the set of Riemannian metrics on a manifold whose extremal properties have also been studied in various settings, \cite{OPS1}.  The value of a spectral zeta function at specific points is further used in physics to define quantities like Casimir energy \cite{casimir}.  In that context, formally the sum $\sum_{n\geq 1} \lambda_n^{1/2}$ comes up in energy computations leading to an analysis of $\zeta (z)$ about $z=-1/2$ \cite{kkbook}. Residues and values at various points connect the eigenvalue spectrum with the geometry and topology of the underlying manifold \cites{gilkey95,gilkey04}.  Determinants of Laplace operators are at the core of the relationship between different areas of mathematics, such as geometric analysis, spectral theory, mathematical physics and number theory.

In geometric analysis and dynamical systems, there are even more zeta functions, including the celebrated Selberg zeta function \cite{selberg} as well as Ruelle and more general dynamical zeta functions \cite{ruelle}.  These functions are useful for demonstrating relationships between the Laplace spectrum on a Riemannian manifold and its geo\-metry, especially the lengths of its closed geodesics.  Dynamical zeta functions are often used to study closed orbits in flows, including the geodesic flow.  The proof of the prime number theorem for the distribution of primes based on the complex analysis of the Riemann zeta function helped to inspire the proof of so-called `prime orbit theorems' for the distribution of prime orbits of a dynamical system based on the analysis of its associated dynamical zeta function \cites{rowlett_pjm, rowlett_err, rowlett_aif, rsst}.

All of these different types of zeta functions typically share a few common features, like their defining series converges in a certain half space of the complex plane.  In many cases, the series admits a meromorphic continuation to the entire complex plane.  When this is the case, it is often interesting to determine the value of the zeta function and its derivative at the origin, if this point is not a pole.  This motivates our present study.  We focus on a certain one-parameter family of Barnes type zeta functions \cite{barnes}. The Barnes zeta function
\beq \zeta_{\mathcal B, N} (z, w | a_1, \ldots, a_N) = \sum_{n_1, \ldots, n_N \geq 0} \frac{1}{(w+n_1 a_1 + \ldots + n_N a_N)^z} \nn \eeq
is defined for $w, a_j \in \C$ such that their real parts are positive for all $j$, and for complex $z$  with real part $\re(z) > N$. 
Specializing to the case $N=2$, we define
\beq
\zeta_{\mathcal B} (z; a,b,x) = \sum_{m,n=0}^\infty \left( am+bn+x\right)^{-z}.\nn
\eeq
Here, $a, b, x \in \C$ have positive real parts, and this expression is well-defined for $z \in \C$ with real part larger than $2$.  Barnes developed a comprehensive theory for a new class of special functions, which he called multiple zeta and gamma functions \cite{barnes}.  These zeta functions are related to Selberg zeta functions and the form factor program for integrable field theories \cites{smirnov, lukyanov} and in studies of XXZ model correlation functions \cite{xxz}.
The one-parameter family of Barnes type zeta functions which is the first main protagonist in our study is
\begin{equation} \label{eq:def_zetac} \zeta_{c} (z) = \sum_{n, \ell=1} ^\infty (c \ell + n)^{-z} = c^{-z} \left( \zeta_\cB (z; c^{-1}, 1, 1) - \zeta_R (z) \right), \quad \Rea c > 0. \end{equation}
This zeta function is well-defined for $z \in \C$ with real part larger than $2$. It is this zeta function one encounters in the analysis of the Laplacian on circular sectors.  In our first result, Proposition \ref{prop:zeta_c} we calculate $\zeta_c '(0)$.  We then proceed to investigate the dependence on the parameter $c$.  We show two different methods to compute the derivative with respect to $c$ of $\zeta_c '(0)$.  The first method is to simply differentiate the expression obtained in Proposition \ref{prop:zeta_c} with respect to $c$ and show that this is well-defined.  Interestingly, there is a completely independent method based on geometric and microlocal analysis \cites{AldRow, AldRowE} that further involves the second class of zeta function protagonists in our study, the Bessel zeta functions.  Let $\lambda_{n,0}$ denote the $n^{th}$ positive zero of the Bessel function $J_0$.  Then we define the Bessel zeta function
\beq  \xi_0 (z) := \sum_{n \geq 1} \lambda_{n, 0} ^{-z}.  \label{eq:besselzeta0} \eeq
More generally, for $\Rea c > 0$, let $\lambda_{n, \ell}$ denote the $n^{th}$ positive zero of the Bessel function $J_{c \ell}$ of order $c \ell$.  Then, we define the Bessel zeta function
\beq \xi_c (z) = \sum_{n, \ell = 1} ^\infty \lambda_{n, \ell} ^{-z}. \label{eq:bessel_zeta} \eeq
Studies of such Bessel zeta functions seem to go back to Hawkins \cite{hawk} with several more details provided in
\cite{actben}.  The first main result of this work is the following theorem for the Barnes zeta function.  

\begin{theorem}\label{th:main1} 
Let $\zeta_c(z)$ be defined as in \eqref{eq:def_zetac}.  Then the derivative with respect to $c$ of $\zeta_c'(0)$ is 
\beq \frac{d}{dc} \zeta_c '(0) &=& \int_1 ^\infty \frac 1 t \frac{1}{e^t - 1} \frac{-t e^{jt}}{(e^{jt} -1)^2} dt \nn \\
&+& \int_0 ^1 \frac 1 t \left( \frac{-t e^{jt}}{(e^t -1)(e^{jt}-1)^2} + \frac{1}{j^2 t^2} - \frac{1}{2j^2 t} - \frac 1 {12} + \frac{1}{12 j^2} \right) dt \nn \\
&+& \gamma_e \left( \frac{1}{12} - \frac{1}{12j^2} \right). \label{eq:th1_eq1} \eeq 
Here $\psi(x) = \frac{d}{dx} \log \Gamma(x)$ is the so-called di-Gamma function.  If $c=j>1$ is an integer, then this expression is equal to 
\beq \left . \frac{d}{dc} \zeta_c '(0) \right|_{c=j} &=& - \frac{1}{12} - \frac{1}{8j^2} + \frac{1-j^2}{12 j^2} \log j -  \frac{1}{2 j^3}\sum_{p=1} ^j p(j-p) \psi \left( \frac p j \right). \label{eq:th1_eq2}
\eeq
\end{theorem}

\begin{remark}  For certain experts, the first expression for the variation of the Barnes zeta function with respect to the parameter may be expected or known.  However, for a more broad audience of mathematicians and physicists, based on the definition \eqref{eq:def_zetac} as an infinite series, it may not be so immediate that the variation can be expressed as a sum of integrals without any infinite series.  Moreover, to the best of our knowledge, our derivation in \S \ref{ss:barnesj} of the expression in \eqref{eq:th1_eq2} in the case when the parameter is an integer presented here is novel.  
\end{remark}

Our next main result concerns the Bessel zeta function. 

\begin{theorem} \label{th:main2} 
Let $\xi_c(z)$ be defined as in \eqref{eq:bessel_zeta}.  Assume that the parameter $c>1$ is an integer.  Then we have the equality for the derivative with respect to $c$ of the Bessel zeta function $\xi_c '(0)$
\beq \frac{d}{dc} \left(\xi_c '(0)\right) &=& \frac 1 2 \int_1 ^\infty \frac 1 t \frac{1}{e^t - 1} \frac{-t e^{ct}}{(e^{ct} -1)^2} dt \nn\\
& &+ \frac 1 2  \int_0 ^1 \frac 1 t \left( \frac{-t e^{ct}}{(e^t -1)(e^{ct}-1)^2} + \frac{1}{c^2 t^2} - \frac{1}{2c^2 t} - \frac 1 {12} + \frac{1}{12 c^2} \right) dt  \nn\\
& & + \frac 1 2 \gamma_e \left( \frac{1}{12} - \frac{1}{12c^2} \right) - \frac{5}{48c^2} - \frac{1}{24} \log(2) \left( 1 - \frac 1 {c^2} \right) \nn \\
&=&- \frac{\pi}{2 c^2 } \left( \frac{1}{3\pi} + \frac{c^2}{12 \pi} - \frac{\gamma_e}{12\pi} \left( c^2 - 1 \right) - \frac{1}{2\pi} \sum_{k=1} ^{\ceil*{ \frac c 2 -1}} \frac{ \log | \sin(k \pi/c)|}{\sin^2(k \pi/c)} \right). \label{eq:th2_eq1} 
\eeq 
Assume that the parameter $c>1$ is not an integer.  Then we have the equality for the derivative with respect to $c$ of the Bessel zeta function $\xi_c '(0)$
\beq \frac{d}{dc} \left(\xi_c '(0)\right) &=& \frac 1 2 \int_1 ^\infty \frac 1 t \frac{1}{e^t - 1} \frac{-t e^{ct}}{(e^{ct} -1)^2} dt \nn\\
& &+ \frac 1 2  \int_0 ^1 \frac 1 t \left( \frac{-t e^{ct}}{(e^t -1)(e^{ct}-1)^2} + \frac{1}{c^2 t^2} - \frac{1}{2c^2 t} - \frac 1 {12} + \frac{1}{12 c^2} \right) dt  \nn\\
& & + \frac 1 2 \gamma_e \left( \frac{1}{12} - \frac{1}{12c^2} \right) - \frac{5}{48c^2} - \frac{1}{24} \log(2) \left( 1 - \frac 1 {c^2} \right) \nn \\
&=& - \frac{\pi}{2c^2} \left( \frac{1}{3\pi} + \frac{c^2}{12 \pi} +  \sum_{k = \lceil{-\frac c 2 \rceil}, k \neq 0 } ^{\lceil \frac c 2 - 1 \rceil} \frac{-2\gamma_e + \log 2 - \log\left({1-\cos(2k\pi/c)}\right) }{4 \pi (1-\cos(2k\pi/c))} \right) \nn \\ 
& & - \frac{\pi}{2c^2} \left( \frac{2c}{\pi} \sin(\pi c) \int_\R \frac{ - \log 2 + 2 \gamma_e + \log(1+\cosh(s))}{16 \pi (1+\cosh(s)) (\cosh(c s) - \cos(c \pi))} ds \right). \label{eq:th2_eq2} 
\eeq 
\end{theorem} 

\begin{remark}  A similar remark is in order here, as some experts may either know or expect the first expression given here for the variation of the Bessel zeta function with respect to the parameter.  The reformulations in the cases in which $c$ is an integer (or not) are however based on identities obtained here of which we are not aware from the literature.  
\end{remark}

\subsection{Organization} 
This work is organized as follows. In \S \ref{s:zetac} we prove Theorem \ref{th:main1}.  To do this, we first calculate the Barnes zeta function's derivative at zero, $\zeta_c '(0)$ as well as its derivative with respect to the parameter $c$, thereby completing the proof of \eqref{eq:th1_eq1} in Theorem \ref{th:main1}.  We then focus on the case in which the parameter $c$ is a natural number, completing the proof of Theorem \ref{th:main1} in \S \ref{ss:barnesj}. In \S \ref{s:bessel_zeta} we prove Theorem \ref{th:main2}.  We begin by showing that the Bessel zeta functions $\xi_0$ and $\xi_c$ are holomorphic in a neighborhood of $0$, and we compute their derivatives at $0$.  We then compute the variation with respect to the parameter $c$, in Proposition \ref{prop:xic}, obtaining the first expressions in \eqref{eq:th2_eq1} and \eqref{eq:th2_eq2}.  There is an alternative way to compute the variation with respect to the parameter $c$ based on geometric and microlocal analysis which we give in Proposition \ref{prop:v2}.  We then complete the proof of Theorem \ref{th:main2} by demonstrating several identities involving both elementary and special functions that may be of independent use or interest. We note that the proofs of \eqref{eq:th2_eq1} and \eqref{eq:th2_eq2} require quite different techniques in these two cases.  We conclude in \S \ref{s:conclude} with further directions and possible applications.

\section{The Barnes zeta function} \label{s:zetac} 
Before we start with the computation of the derivative at zero of the Barnes zeta function and its variation with respect to the parameter $c$, let us recall the definition of the `big-$\mathcal O$' notation.
\begin{defn} \label{def:bigO}  We write that $f(t)$ is $\mathcal O(g(t))$ as $t \to t_0 \in \C$ if there is a constant $C>0$ and $\epsilon > 0$ such that for all $|t-t_0| < \epsilon$,
\[ |f(t)| \leq C g(t).\]
We write that $f(t)$ is $\mathcal O(g(t))$ as $|t| \to \infty$ if there is a constant $M>0$ such that the above inequality holds for all $|t| > M$.
\end{defn}

We further recall the $\Gamma$ function,
\beq \Gamma(z) = \int_0 ^\infty t^{z-1} e^{-t} dt, \quad \Rea (z) > 0. \label{eq:Gamma} \eeq
By the functional equation $\Gamma(z+1) = z \Gamma(z)$, one obtains the meromorphic extension to $z \in \C$ with simple poles at non-positive integer points.  For the sake of clarity we note that $\log$ here is the real natural logarithm, sometimes written $\ln$.

\begin{prop} \label{prop:zeta_c}
Let $\zeta_c(z)$ be defined as in \eqref{eq:def_zetac}.  Define
\beq
b_{-2} = \frac 1 c , \quad b_{-1} = - \frac 1 2 - \frac{1}{2c}, \quad b_0 = \frac 1 4 + \frac{c}{12} + \frac{1}{12 c}. \label{eq:b012}
\eeq
Then,
\beq
\zeta_{c} ' (0) = \int\limits_1^\infty  \frac 1 t \,\,\frac 1 {e^{ct } -1} \,\, \frac 1 {e^t -1}  dt +\int\limits_0^1  \frac 1 t \,\,\left(\frac 1 {e^{ct } -1} \,\, \frac 1 {e^t -1} - \frac{ b_{-2}}{t^2} - \frac{ b_{-1}} t - b_0 \right) dt - \frac 1 2 b_{-2} - b_{-1} + b_0 \gamma_e. \nn
\eeq
\end{prop}

\begin{proof}
We have
\[ \Gamma(z) \zeta_c(z) = \int_0 ^\infty x^{z-1} e^{-x} dx \sum_{n, \ell \geq 1} (c\ell+n)^{-z}, \quad \Rea(z) > 2. \]
By the absolute convergence we may switch summation and integration and make the substitution $t = \frac{x}{c\ell+n}$ to obtain
\[ \Gamma(z) \zeta_c(z) = \int_0 ^\infty t^{z-1} \sum_{n, \ell \geq 1} e^{-(c\ell+n)t} dt = \int_0 ^\infty t^{z-1} \frac{1}{e^{ct} - 1} \frac{1}{e^t  - 1} dt. \]
The last equality follows from summing the two geometric series.  We therefore obtain
\begin{equation} \label{sec8}
\zeta_{c} (z) = \frac 1 {\Gamma (z)} \int\limits_0^\infty  t^{z-1} \,\, \frac 1 {e^{ct } -1} \,\, \frac 1 {e^t -1} dt.  \end{equation}
From the $t\to 0$ behavior it follows that this representation is valid for $\Rea z>2$. We calculate the Laurent expansion near $t=0$ of the exponential terms
\[ \frac 1 {e^{c t} -1} \,\, \frac 1 {e^t-1} = \frac{b_{-2}} {t^2} + \frac{b_{-1}} t + b_0 + \mathcal O(t), \]
with the coefficients $b_0$, $b_1$, and $b_2$ defined in \eqref{eq:b012}. We therefore rewrite (\ref{sec8}) as
\beq
\zeta_{c} (z) &=& \frac 1 {\Gamma (z) } \int\limits_1 ^\infty t^{z-1} \,\, \frac 1 {e^{ct } -1} \,\, \frac 1 {e^t -1} dt + \frac 1 {\Gamma (z) } \int\limits_0 ^1  t^{z-1} \,\, \left( \frac 1 {e^{ct } -1} \,\, \frac 1 {e^t -1} - \frac{b_{-2}} {t^2} - \frac{b_{-1}} t - b_0\right) dt \nn\\
& &+ \frac 1 {\Gamma (z) } \int\limits_0 ^1  t^{z-1} \,\, \left(\frac{b_{-2}} {t^2} + \frac{b_{-1}} t + b_0\right) dt = \frac 1 {\Gamma (z) } \int\limits_1 ^\infty t^{z-1} \,\, \frac 1 {e^{ct } -1} \,\, \frac 1 {e^t -1} dt \nn\\
& &+ \frac 1 {\Gamma (z) } \int\limits_0 ^1  t^{z-1} \,\, \left(\frac 1 {e^{ct } -1} \,\, \frac 1 {e^t -1} - \frac{b_{-2}} {t^2} - \frac{b_{-1}} t - b_0\right) dt +\frac 1 {\Gamma (z)} \left( \frac{b_{-2}} {z-2} + \frac{b_{-1}} {z-1} + \frac{b_0} z \right).\label{sec9}
\eeq
This is now valid for $\Rea z > -1$, in particular suitable for the computation of $\zeta _{c} ' (0)$.  Since $\Gamma(z)$ has a simple pole at $z=0$ with residue equal to one, we calculate
\[ \left . \frac{d}{dz} \left( \frac{1}{\Gamma(z)} \right) \right|_{z=0} = 1, \quad \left . \frac{d}{dz} \left( \frac{1}{z\Gamma(z)} \right) \right|_{z=0} = \gamma_e. \]
Above, $\gamma_e$ is the Euler-Mascheroni constant.  With this we obtain the proposition's expression for $\zeta_c'(0)$. 
\end{proof}

As a corollary, we obtain the derivative with respect to the parameter $c$.
\begin{cor} \label{cor:dc_zetac}
Let $\zeta_c(z)$ be defined as in \eqref{eq:def_zetac}.  Then the derivative with respect to $c$ of $\zeta_c'(0)$ is
\beq \frac{d}{dc} \zeta_c '(0) &=& \int_1 ^\infty \frac 1 t \frac{1}{e^t - 1} \frac{-t e^{ct}}{(e^{ct} -1)^2} dt \nn \\ 
&+&  \int_0 ^1 \frac 1 t \left( \frac{-t e^{ct}}{(e^t -1)(e^{ct}-1)^2} + \frac{1}{c^2 t^2} - \frac{1}{2c^2 t} - \frac 1 {12} + \frac{1}{12 c^2} \right) dt + \gamma_e \left( \frac{1}{12} - \frac{1}{12c^2} \right). \label{cor1}
\eeq
\end{cor}
\begin{proof}
The proof is a consequence of differentiating the terms in $\zeta_c'(0)$ with respect to $c$ which is justified by the absolute convergence of the integrals together with the dominated convergence theorem. We note that
\[ \frac{d}{dc} \left(b_{-2}\right) = \frac {-1}{ c^2} , \quad \frac{d}{dc}\left( b_{-1} \right) =  \frac{1}{2c^2}, \quad \frac{d}{dc} \left( b_0 \right) = \frac{1}{12} - \frac{1}{12 c^2}. \]
\end{proof}

\subsection{The variation of the Barnes zeta function when the parameter is a natural number} \label{ss:barnesj}
To prove the expression for the variation of $\zeta_c '(0)$ with respect to $c$ when $c$ is an integer given in \eqref{eq:th1_eq2} of Theorem \ref{th:main1}, we demonstrate identities involving both elementary and special functions.  For this reason, this method and the identities involving both elementary and special functions obtained along the way may be of independent interest.

\begin{prop} \label{prop:barnes1} The derivative with respect to the parameter $c$, of $\zeta_c '(0)$, evaluated at an integer point $c=j \in \N$ is
\[ \left . \frac{d}{dc} \zeta_c '(0) \right|_{c=j} =  - \frac{1}{12} - \frac{1}{8j^2} + \frac{1-j^2}{12 j^2} \log j -  \frac{1}{2 j^3}\sum_{p=1} ^j p(j-p) \psi \left( \frac p j \right). \]

\end{prop}
\begin{proof}
Differentiating the definition of the Barnes zeta function with respect to the parameter $c$ and setting $c=j \in \N$ yields
\[ \left . \frac{d}{dc} \zeta_c(s) \right|_{c=j} =\left.  -s \sum_{n \geq 1} \sum_{\ell \geq 1} \ell (\ell c + n)^{-s-1} \right|_{c=j}= -s \sum_{n \geq 1} \sum_{\ell \geq 1} \ell (\ell j + n)^{-s-1}. \]

We split the sum over $n$ into sums over the congruency classes of $\Z / j \Z$,
\[ \sum_{n \geq 1} \sum_{\ell \geq 1} \ell (j\ell +n)^{-s-1} = \sum_{p=1} ^j \sum_{m \geq 0} \sum_{\ell \geq 1} \ell (j\ell + mj+p)^{-s-1}, \quad n=mj+p. \]
We consider each of these sums separately for $p \in \{1, \ldots, j\}$,
\[\sum_{m \geq 0} \sum_{\ell \geq 1}  \ell (j\ell + mj+p)^{-s-1} = \sum_{m \geq 0} \sum_{\ell \geq 1} \ell j^{-s-1} \left( \ell + m + \frac p j \right)^{-s-1}.\]
We make the substitution,  $k = \ell + m$.  Then $k$ ranges from $1$ to $\infty$, but the sum over $\ell$ is only from $1$ to $k$, so we obtain
\[\sum_{m \geq 0} \sum_{\ell \geq 1} \ell j^{-s-1} \left( \ell + m + \frac p j \right)^{-s-1} =  j^{-s-1} \sum_{k \geq 1} \sum_{\ell=1} ^k \left( k + \frac p j \right)^{-s-1} \ell. \]
Using the well known identity $\sum_{\ell=1} ^k \ell = \frac{k(k+1)}{2}$, we have
\beq   j^{-s-1} \sum_{k \geq 1} \sum_{\ell=1} ^k \left( k + \frac p j \right)^{-s-1} \ell  
&=&  \frac 1 2 j^{-s-1} \sum_{k\geq 1} \left( k + \frac p j \right)^{-s-1} k (k+1) \nn \\
&=& \frac 1 2 j^{-s-1} \sum_{k \geq 1} \left( k + \frac p j\right)^{-s-1} \left( k+ \frac p j - \frac p j \right) \left( k + \frac p j + \frac{j-p}{j} \right) \nn \\
& = & \frac 1 2 j^{-s-1} \sum_{k \geq 1}  \left( k + \frac p j \right)^{-s-1} \left[ \left( k + \frac p j \right)^2 + \left( \frac{j-p}{j} - \frac p j \right) \left( k + \frac p j \right) - \frac{p(j-p)}{j^2} \right].  \nn \eeq
We separate these terms by using the Hurwitz zeta function,
\[ \zeta_H (s;q) := \sum_{k \geq 0} (k+q)^{-s}.\]
We then obtain that
\[ \left . \frac{d}{dc} \zeta_c (s)\right|_{c=j} = - \frac{sj^{-s-1}}{2} \sum_{p=1} ^j \left\{ \zeta_H \left( s-1; \frac p j \right) + \frac{j-2p}{j} \zeta_H \left( s; \frac p j \right) - \frac{p(j-p)}{j^2} \zeta_H \left( s+1; \frac p j \right) \right\} \]
\[ = \RN{1} + \RN{2} + \RN{3}.\]

Now we differentiate each of these terms separately.  For the first two terms, we simply compute the derivative with respect to $s$ and then set $s=0$. The $s$ in the front makes this computation very easy.  The third term in the bracket has a singularity at $s=0$. Thus, there will be a contribution from the residue at $s=0$ to the derivative. Here are the computations term by term:
\[ \left . \frac{d}{ds} \right|_{s=0} \RN{1} = - \frac{1}{2j} \sum_{p=1} ^j \zeta_H \left( -1; \frac p j \right) = - \frac{1}{2j} \sum_{p=1} ^j (-1) \frac{B_2 (p/j)}{2} \]
by \cite[9.531]{gr},  where here and in what follows $B_n$ denotes the $n$th Bernoulli polynomial.  By \cite[9.627]{gr}, $B_2 (x) = x^2 - x + \frac 1 6$, and so we have
\[ \frac{1}{4j} \sum_{p=1} ^j B_2(p/j) = \frac{1}{4j} \sum_{p=1} ^j \left( \frac 1 6 - \frac p j + \frac{p^2}{j^2} \right) = \frac{1}{4j}  \left( \frac j 6 - \frac{j(j+1)}{2j} + \frac{j (j+1)(2j+1)}{6j^2} \right)  = \frac{1}{24j^2}. \]
Next,
\[ \left . \frac{d}{ds} \right|_{s=0} \RN{2} = - \frac{1}{2j} \sum_{p=1} ^j \frac{j-2p}{j} \zeta_H \left( 0; \frac p j\right) = - \frac{1}{2j} \sum_{p=1} ^j \frac{j-2p}{j} (-1) B_1(p/j),\]
by \cite[9.531]{gr}.  By \cite[9.627.1]{gr}, $B_1 (x) = x - \frac 1 2$, and so this is
\[ \frac{1}{2j^2} \sum_{p=1} ^j (j-2p)(p/j-1/2) \] 
\[=  \frac{1}{2j^2} \sum_{p=1} ^j (p - 2p^2/j - j/2 + p) = \frac{1}{2j^2} \left( j(j+1) - \frac{j(j+1)(2j+1)}{3j} - \frac{j^2}{2} \right) = \frac{1}{2j^2} \left( j^2 + j - \frac{2j^2 + 3j+1}{3} - \frac{j^2}{2} \right)\]
\[ = \frac{1}{2j^2} \left( - \frac{j^2}{6} - \frac 1 3 \right) = - \frac{1}{12} - \frac{1}{6j^2}.\]
Since $\RN{3}$ is singular at $s=0$, we calculate the asymptotic behavior as $s\to 0$ of the terms.  The third term is
\[ \frac{s j^{-s-3}}{2} \sum_{p=1} ^j p(j-p) \zeta_H \left( s+1; \frac p j \right). \]
By \cite[9.533.2]{gr} we have that
\[ \zeta_H \left( s+1; \frac p j \right) = \frac 1 s - \psi \left( \frac p j \right) + \cO(s), \quad s \to 0, \quad \psi (x) = \frac{d}{dx} \log \Gamma(x). \]
Since
\[ j^{-s-3} = \frac 1 {j^3} e^{-s \log j} = \frac{1}{j^3} \left( 1 - s \log j + \cO(s^2) \right), \]
we therefore have as $s \to 0$
\[ \frac{s j^{-s-3}}{2} \zeta_H \left( s+1; \frac p j \right) = \frac{s}{2j^3} \left( 1 - s \log j + \cO(s^2) \right) \left( \frac 1 s - \psi \left( \frac p j \right) + \cO(s) \right) \]
\[ = \frac{s}{2j^3} \left( \frac 1 s - \psi \left( \frac p j \right) - \log j + \cO(s) \right). \]
Consequently we compute that
\[\frac{d}{ds} \left . \RN{3} \right|_{s=0} = \frac{1}{2 j^3} \sum_{p=1} ^j p(j-p) \left( - \psi \left( \frac p j \right)  - \log j \right) =  - \frac{1}{2 j^3} \frac{j (j+1)(j-1)}{6} \log j - \frac{1}{2 j^3}\sum_{p=1} ^j p(j-p) \psi \left( \frac p j \right)\]
\[ = \frac{1-j^2}{12 j^2} \log j -   \frac{1}{2 j^3}\sum_{p=1} ^j p(j-p) \psi \left( \frac p j \right).\]
Putting together the three terms we obtain
\[ \left . \frac{d}{dc} \zeta_c '(0) \right|_{c=j} =  \frac{1}{24j^2} - \frac{1}{12} - \frac{1}{6j^2} + \frac{1-j^2}{12 j^2} \log j -  \frac{1}{2 j^3}\sum_{p=1} ^j p(j-p) \psi \left( \frac p j \right) \]
\[ = - \frac{1}{12} - \frac{1}{8j^2} + \frac{1-j^2}{12 j^2} \log j -  \frac{1}{2 j^3}\sum_{p=1} ^j p(j-p) \psi \left( \frac p j \right). \]
\end{proof}

\section{The Bessel zeta functions} \label{s:bessel_zeta}
We first demonstrate a result for the Bessel zeta function, $\xi_0$ as defined in \eqref{eq:besselzeta0}.

\begin{prop} \label{prop:zero_momentum}
The Bessel zeta function $\xi_0(s)$ in \eqref{eq:besselzeta0} is holomorphic in a neighborhood of $s=0$, and
\[ \xi_0 '(0) = - \frac 1 2 \log(2\pi).\]
\end{prop}

\begin{proof}
By \cite[p. 57]{kkbook},
\beq \xi_0 (s) = \int_\gamma \frac{1}{2\pi i} k^{-2s} \frac{d}{dk} \log J_0 (k) dk = \frac{\sin \pi s}{\pi} \int_0 ^\infty  z^{-2s} \frac{d}{dz} \log I_0 (z) dz, \quad \frac 1 2 < \Rea(s) < 1. \nn \eeq
Here $\gamma$ is a contour that goes around the positive part of the real axis while enclosing all $\lambda_{n,0} > 0$. The precise form of the contour is given in \cite[p. 43]{kkbook} but it is not relevant for our purposes.   Above, $I_0$ is the modified Bessel function of the first type of order zero. The behaviours as $z \to 0$ and $\infty$ of the modified Bessel function are, respectively,
\beq I_0 (z) = 1 + \frac{z^2}{4} + \cO(z^4), \quad  z \to 0, \label{eq:I0_zero} \eeq
\beq  \log I_0 (z) \sim z - \frac 1 2 \log(2\pi z) + \cO(1/z), \quad z \to \infty. \label{eq:I0_infty} \eeq
We split the integral above as 
\[ \xi_0 (s) = \frac{\sin \pi s}{\pi} \left( \int_0 ^1 + \int_1 ^\infty \right) z^{-2s} \frac{d}{dz} \log I_0 (z) dz = \frac{\sin \pi s}{\pi}  \int_0 ^1 z^{-2s} \frac{d}{dz} \log I_0 (z) dz \label{eq:zm_term1} \] 
\[ + \frac{\sin \pi s}{\pi} \int_1 ^\infty  z^{-2s} \frac{d}{dz} \left( \log I_0 (z) - z + \frac 1 2 \log(2\pi z) \right) dz  +   \frac{\sin \pi s}{\pi} \int_1 ^\infty  z^{-2s} \frac{d}{dz} \left( z - \frac 1 2 \log(2\pi z) \right) dz.  \] 
This last integral becomes 
\[  \int_1 ^\infty  z^{-2s}  \left( 1 - \frac{1}{2z} \right) dz = \frac{1}{2s-1} - \frac{1}{4s}, \quad \Rea s > \frac 1 2. \]
For the middle term, we use integration by parts to obtain
\[ \int_1 ^\infty  z^{-2s} \frac{d}{dz} \left( \log I_0 (z) - z + \frac 1 2 \log(2\pi z) \right) dz  \] 
\[ = \left . z^{-2s} \left( \log I_0 (z) - z + \frac 1 2 \log(2\pi z) \right)\right|_1 ^\infty 
+ \int_1 ^\infty 2s z^{2s-1} \left( \log I_0 (z) - z + \frac 1 2 \log(2\pi z) \right) dz. \] 

By \eqref{eq:I0_infty}, this is well-defined for $\Rea(s) > - \frac 1 2$, and it is
\[ - \log I_0 (1) + 1 - \frac 1 2 \log(2\pi) + 2s \int_1 ^\infty z^{2s-1} \left( \log I_0 (z) - z + \frac 1 2 \log(2\pi z) \right) dz. \]
By \eqref{eq:I0_zero}, the first term \eqref{eq:zm_term1} is well-defined for $\Rea(s) < 1$.  We therefore obtain
\[  \xi_0 (s) =  \frac{\sin \pi s}{\pi} \int_0 ^1 z^{-2s} \frac{d}{dz} \log I_0 (z) dz +   \frac{\sin \pi s}{\pi} \left[ - \log I_0 (1) + 1 - \frac 1 2 \log(2\pi) \right] \]
\[+ \frac{2s \sin \pi s}{\pi}  \int_1 ^\infty z^{2s-1} \left( \log I_0 (z) - z + \frac 1 2 \log(2\pi z) \right) dz \nn +  \frac{\sin \pi s}{\pi} \left( \frac{1}{2s-1} - \frac{1}{4s} \right). \]
The first term is well defined for all $s$ with $\Rea(s) < 1$.  We obtained the second and third terms for $\Rea(s) > - \frac 1 2$.  The fourth term is valid for $\frac 1 2 < \Rea(s)$ and admits a meromorphic extension to $-\frac 1 2 < \Rea(s)$ with a simple pole at $s= \frac 1 2$.  This extension is holomorphic in a neighborhood of $s=0$ due to the vanishing of the pre-factor, $\sin(\pi s)$. We therefore compute:
\[ \xi_0 ' (0) = \left . \frac{d}{ds} \frac{ \sin \pi s}{\pi} \right|_{s=0} \left( \log I_0 (1) - \log I_0 (0) - \log I_0 (1) + 1 - \frac 1 2 \log(2\pi) -1 \right)  \overset{I_0(0)=1}=  - \frac 1 2 \log(2\pi). \] 
\end{proof}

We next demonstrate a result for the general Bessel zeta functions $\xi_c(s)$ defined in \eqref{eq:bessel_zeta}.
\begin{prop} \label{prop:xic}
The Bessel zeta function $\xi_c(s)$ is holomorphic in a neighborhood of $s=0$ and satisfies
\beq \xi_c '(0) &=& \frac 1 2 \int\limits_1^\infty  \frac 1 t \,\,\frac 1 {e^{ct } -1} \,\, \frac 1 {e^t -1}  dt + \frac 1 2 \int\limits_0^1  \frac 1 t \,\,\left(\frac 1 {e^{ct } -1} \,\, \frac 1 {e^t -1} - \frac{ b_{-2}}{t^2} - \frac{ b_{-1}} t - b_0 \right) dt \nn\\
& & + \frac 1 2 \left( \frac 1 2 + \gamma_e\left( \frac 1 4 + \frac{c}{12} + \frac{1}{12 c} \right) \right)  + \frac{5}{48c} - \frac{1}{24} \left( c + \frac 1 c \right)\log 2 . \nn \eeq
Moreover, the variation with respect to the parameter $c$ is
\beq \frac{d}{dc} \left(\xi_c '(0)\right) &=& \frac 1 2 \int_1 ^\infty \frac 1 t \frac{1}{e^t - 1} \frac{-t e^{ct}}{(e^{ct} -1)^2} dt + \frac 1 2  \int_0 ^1 \frac 1 t \left( \frac{-t e^{ct}}{(e^t -1)(e^{ct}-1)^2} + \frac{1}{c^2 t^2} - \frac{1}{2c^2 t} - \frac 1 {12} + \frac{1}{12 c^2} \right) dt  \nn\\
& & + \frac 1 2 \gamma_e \left( \frac{1}{12} - \frac{1}{12c^2} \right)  - \frac{5}{48c^2} - \frac{1}{24} \left( 1 - \frac 1 {c^2} \right) \log 2 . \nn \eeq
\end{prop}

\begin{proof}
We have the identity from \cite{bdk_96}
\[ \xi_c(s) = \sum_{ \ell = 1} ^\infty \frac{1}{2\pi i} \int_\gamma k^{-s} \frac{\pa}{\pa k} \log J_{c\ell} (k) dk. \]
The contour $\gamma$ encloses all of the positive zeros $\lambda_{n, \ell}$, and the representation is valid for $\Rea(s) > 1$.  The analytical continuation has been constructed in detail and in greater generality in \cite{bdk_96} which show that is holomorphic in a neighborhood of $s=0$ and (see also \cite[\S 3]{akr_pre})
\beq \xi_c '(0) = \frac 1 2 \left[ \zeta_{c} '(0) + \frac{5}{24c} - \frac{1}{12} \left( c + \frac 1 c \right) \log 2  \right].  \label{eq:xic_zetac} \eeq
Substituting the values from Proposition \ref{prop:zeta_c} and Corollary \ref{cor:dc_zetac} completes the proof.
\end{proof}

The variation of the Bessel zeta function with respect to the parameter $c$ can be obtained as a corollary to \cite[Theorem 4]{AldRow}.\footnote{There is unfortunately a typo in the statement of the theorem \cite[Theorem 4]{AldRow} that has been corrected here as well as in \cite{akr_pre}.  We note that this is mere a transcription error from the contents of the proof.  Further, the boundary contribution was overlooked in \cite{AldRow}, but has been corrected in  \cite{AldRowE}.} We note that the proofs of Proposition \ref{prop:xic} and that of Proposition \eqref{prop:v2} are completely independent. 

\begin{prop} \label{prop:v2} The derivative of the Bessel zeta function $\xi_c'(0)$ with respect to the parameter $c$ is, for $c \in (1, \infty)$,
\[  \frac{d}{dc} \xi_c '(0) = - \frac{\pi}{2c^2} \left( \frac{1}{3\pi} + \frac{c^2}{12 \pi} +  \sum_{k = \lceil{-\frac c 2 \rceil}, k \neq 0 } ^{\lceil \frac c 2 - 1 \rceil} \frac{-2\gamma_e + \log 2 - \log\left({1-\cos(2k\pi/c)}\right) }{4 \pi (1-\cos(2k\pi/c))} \right) \]
\[ - \frac{\pi}{2c^2} \left( \frac{2c}{\pi} \sin(\pi c) \int_\R \frac{ - \log 2 + 2 \gamma_e + \log(1+\cosh(s))}{16 \pi (1+\cosh(s)) (\cosh(c s) - \cos(c \pi))} ds \right). \]
\end{prop}

\begin{proof}
The Bessel zeta function evaluated at $2s$, $\xi_c(2s)$ is equal to the spectral zeta function for the Laplace operator with the Dirichlet boundary condition on a circular sector of radius one and opening angle $\frac \pi c$, for $c>1$ real.    Denoting this spectral zeta function by $\zeta_{S_\alpha} (s)$ we therefore have
\[ \xi_c(2s) = \zeta_{S_\alpha} (s), \quad \alpha = \frac \pi c, \quad c = \frac \pi \alpha \implies \zeta_{S_\alpha} '(0) = 2 \xi_c '(0), \]
and
\[ \frac{d}{d\alpha} \left( \zeta_{S_\alpha} '(0) \right) = - \frac{\pi}{\alpha^2} \frac{d}{dc} (2 \xi_c '(0)) = - \frac{c^2}{\pi} \frac{d}{dc} (2 \xi_c '(0)) \]
so
\beq \frac{d}{dc} \xi_c '(0) = - \frac{\pi}{2c^2} \frac{d}{d\alpha} \left( \zeta_{S_\alpha} '(0) \right).  \label{eq:2formulas} \eeq
We therefore recall the statement of \cite[Theorem 4]{AldRow}
\[  \frac{d}{d\alpha} \zeta_{S_\alpha} '(0) = \frac{1}{3\pi} + \frac{\pi}{12 \alpha^2} +  \sum_{k\in W_{\alpha}} \frac{-2\gamma_e + \log 2 - \log\left({1-\cos(2k\alpha)}\right) }{4 \pi (1-\cos(2k\alpha))} \]
\[ + \frac{2}{\alpha} \sin(\pi^2/\alpha) \int_\R \frac{ - \log 2 + 2 \gamma_e + \log(1+\cosh(s))}{16 \pi (1+\cosh(s)) (\cosh(\pi s/\alpha) - \cos(\pi^2/\alpha))} ds. \]
Above,
\[ k_{min} = \ceil*{ \frac{-\pi}{2\alpha} }, \textrm{ and } k_{max} = \floor*{\frac{\pi}{2\alpha}} \textrm{ if } \frac{\pi}{2\alpha} \not\in \Z, \textrm{ otherwise } k_{max} = \frac{\pi}{2\alpha} - 1, \]
and
\[ W_{\alpha} =\left\{  k \in \left( \Z \bigcap \left[k_{min},  k_{max}\right]\right) \setminus \left\{ \frac{\ell\pi}{\alpha} \right\}_{\ell \in \Z} \right\}. \]
A straightforward calculation as done in \cite[Proposition 3.2]{akr_pre} shows that 
the set $W_\alpha$ is precisely the set of integers
\[\Z \supset W_\alpha = \{ j \}_{k_{min}} ^{k_{max}} \setminus \{ 0 \}, \quad  k_{max} = \ceil*{ \frac{\pi}{2\alpha} -1}.\]
Using the relationship $c = \frac \pi \alpha$ together with \eqref{eq:2formulas} completes the proof. 
\end{proof}

\subsection{The variation of the Bessel zeta function when the parameter is a natural number} \label{s:main}
Theorem \ref{th:main2} not only states that the expressions in Propositions \ref{prop:xic} and \ref{prop:v2} for the derivative of $\xi_c'(0)$ with respect to $c$ are equal, it also gives a further simplification in the case when the parameter $c$ is a natural number.  Our proof reveals several identities involving special functions, some of which appear to be new.  The first identities we establish by transforming certain sums of trigonometric functions are however not new; they can be obtained from our study of zeta regularized determinants \cite{akr_pre}. 

\begin{prop} \label{prop:reformulate_sum_w_alpha}
For any $\alpha > 0$, for any finite sum over integers that excludes the value $k=0$, we have the equality
\beq \sum_{k} \frac{-2\gamma_e + \log 2  - \log\left({1-\cos(2k\alpha)}\right) }{4 \pi (1-\cos(2k\alpha))} =  -  \sum_{k} \frac{\gamma_e + \log |\sin(k \alpha)|}{4\pi \sin^2 (k \alpha)}. \nn \eeq 
Next, assume that $c>0$ is not an even integer.  Then we have the identity
\[ - \sum_{k = \lceil{-\frac c 2 \rceil}, k \neq 0 } ^{\lceil \frac c 2 - 1 \rceil}  \frac{\gamma_e + \log |\sin(k \pi/c)|}{4\pi \sin^2 (k \pi/c)} = -  \sum_{k=1} ^{\ceil*{\frac c 2-1}}  \frac{\gamma_e + \log |\sin(k \pi/c)|}{2\pi \sin^2 (k \pi/c)}.\]
If $c>0$ is an integer (even or odd), then we have the identities
\beq \sum_{k=1} ^{c-1} \frac{1}{\sin^2 (k \pi/c)} &=& \frac 1 3 \left(c^2-1\right), \nn \\ 
 \sum_{k=1} ^{\ceil*{ \frac c 2 -1}} \frac{ \log |\sin(k \pi/c)|}{2\pi \sin^2 (k \pi/c)} &=& \sum_{k=1} ^{c-1} \frac{ \log |\sin(k \pi/c)|}{4\pi \sin^2 (k \pi/c)}, \nn \\ 
 - \sum_{k = \lceil{-\frac c 2 \rceil}, k \neq 0 } ^{\lceil \frac c 2 - 1 \rceil}  \frac{\gamma_e + \log |\sin(k \pi/c)|}{4\pi \sin^2 (k \pi/c)} &=& - \frac{\gamma_e}{12\pi} \left( c^2 - 1 \right) - \frac{1}{2\pi} \sum_{k=1} ^{\ceil*{ \frac c 2 -1}} \frac{ \log | \sin(k \pi/c)|}{\sin^2(k \pi/c)}. \nn \eeq 
\end{prop}

\begin{proof}
The proof of this result can be obtained by following the argument of \cite[Proposition 3.3]{akr_pre} and using the relationship $c= \frac \pi \alpha$, so $\alpha = \frac \pi c$. 
\end{proof}

The next identity transforms a sum of logarithmic derivatives of the Gamma function into elementary trigonometric functions.  We will use this to obtain the expression for the variation of the Bessel zeta function when the parameter is a natural number given in \eqref{eq:th2_eq2} of Theorem \ref{th:main2}.  

\begin{prop} \label{prop:psi_magic}
For any integer $j \geq 1$, we have the identities
\begin{eqnarray}  \frac{1}{2\pi j} \sum_{p=1} ^j p (j-p) \left( \log(2j) + \psi \left( \frac p j \right) \right)   = - \frac{\gamma_e}{12 \pi} \left( j^2 - 1 \right)  - \frac{1}{2\pi} \sum_{k=1} ^{\left \lfloor \frac{j-1}{2} \right\rfloor}  \frac{ \log \left|\sin\left( \frac{k \pi}{j} \right)\right|}{\sin^2 \left(\frac{k \pi}{j} \right)}, \label{eq:trig_identity0} \\
 \frac{1}{2\pi j} \sum_{p=1} ^{j-1}  p (j-p) \left(- \frac \pi 2 \cot\left( \frac {p \pi}{j} \right) + 2 \sum_{k=1} ^{\left\lfloor \frac{j+1}{2} \right\rfloor - 1} \left[ \cos \left( \frac{2kp \pi}{j} \right) \log \sin \left(\frac {k \pi }{j} \right) \right]\right) \label{eq:trig_identity} \\ 
= - \frac{1}{2\pi} \sum_{k=1} ^{\left \lfloor \frac{j-1}{2} \right\rfloor} \frac{ \log \left|\sin\left( \frac{k \pi}{j} \right)\right|}{\sin^2 \left(\frac{k \pi}{j} \right)},
\nn \\
 \frac{1}{\pi j} \sum_{p=1} ^{j-1}  p (j-p)  \sum_{k=1} ^{\left\lfloor \frac{j+1}{2} \right\rfloor - 1} \left[ \cos \left( \frac{2kp \pi}{j} \right) \log \sin \left(\frac {k \pi }{j} \right) \right]   = - \frac{1}{2\pi} \sum_{k=1} ^{\left \lfloor \frac{j-1}{2} \right\rfloor}  \frac{ \log \left|\sin\left( \frac{k \pi}{j} \right)\right|}{\sin^2 \left(\frac{k \pi}{j} \right)}.  \label{eq:trig_identity2}  \end{eqnarray}
\end{prop}

\begin{proof}
By \cite[8.363.6]{gr}, for $2 \leq j \in \N$ and $\N \ni p \leq j-1$,
\[ \psi\left( \frac p j \right) = - \gamma_e - \log(2j) - \frac \pi 2 \cot\left( \frac {p \pi}{j} \right) + 2 \sum_{k=1} ^{\left\lfloor \frac{j+1}{2} \right\rfloor - 1} \left[ \cos \left( \frac{2kp \pi}{j} \right) \log \sin \left(\frac {k \pi }{j} \right) \right].\]

We observe that
\[ \sum_{p=1} ^j p(j-p) = \sum_{p=1} ^{j-1} p (j-p) = \frac{j (j+1)(j-1)}{6} \implies \frac{1}{2\pi j} \sum_{p=1} ^{j-1}  -\gamma_e p (j-p) = - \frac{\gamma_e}{12\pi} (j^2 - 1). \]

Consequently, establishing \eqref{eq:trig_identity} immediately implies \eqref{eq:trig_identity0}. Consider
\[ \sum_{p=1} ^{j-1} p (j-p) \cot \left( \frac{p \pi}{j} \right).\]
In case $j$ is odd, we split the sum into $\sum_{p=1} ^{\frac{j-1}{2}} + \sum_{p=\frac{j+1}{2}} ^{j-1}$. 
In the second sum, we use the substitution $k=j-p$ to obtain
\[ \sum_{p=1} ^{j-1} p (j-p) \cot \left( \frac{p \pi}{j} \right) = \sum_{p=1} ^{\frac{j-1}{2}} p(j-p) \cot \left( \frac{p \pi}{j} \right) + \sum_{k=1} ^{ \frac{j-1}{2}} (j-k) k \cot\left( \frac{(j-k) \pi}{j} \right) \]
\[ =  \sum_{p=1} ^{\frac{j-1}{2}} p(j-p) \cot \left( \frac{p \pi}{j} \right) + \sum_{k= 1}^{\frac{j-1}{2}} (j-k) k \cot\left( \frac{-k \pi}{j} \right) =0.\]
For $j$ even, we note that
\[p = \frac j 2 \implies \cot\left( \frac{j\pi}{2j} \right) = 0 \implies \sum_{p=1} ^{j-1} = \sum_{p=1} ^{\frac j 2 - 1} + \sum_{p=\frac j 2 + 1} ^{j-1}.\]
Again substituting $k=j-p$ in the second sum we obtain
\[ \sum_{p=1} ^{j-1} p(j-p) \cot \left( \frac{p \pi}{j} \right) = \sum_{p=1} ^{\frac j 2 - 1} p(j-p) \cot \left( \frac{p \pi}{j} \right) + \sum_{k= 1}^{\frac j 2 -1} (j-k) k \cot\left( \frac{(j-k) \pi}{j} \right) =0.\]
It therefore suffices to establish the identity \eqref{eq:trig_identity2} to obtain \eqref{eq:trig_identity} and therefore also \eqref{eq:trig_identity0}.

By \cite[1.352]{gr} 
\[ \sum_{\ell=1} ^{n-1} \ell \cos(\ell x) = \frac{ n \sin \left( \frac{2n-1}{2} x \right)}{2 \sin \left( \frac x 2 \right)} - \frac{1 - \cos(nx)}{4 \sin^2 \left( \frac x 2 \right)}.\] 
In our application, we have $x=\frac{2k\pi}{j}$, letting $n=j$, so
\[  \sum_{p=1} ^{j-1} p \cos \left( \frac{2k\pi}{j} p\right) = \frac{j \sin\left( \frac{2j-1}{2} \frac{2k\pi}{j} \right)}{2 \sin \left( \frac{k\pi}{j} \right)} - \frac{1- \cos\left( j \frac{2k\pi}{j} \right)}{4 \sin^2 \left( \frac{k\pi}{j} \right)}=  \frac{j \sin\left( \frac{2j-1}{j} k\pi \right)}{2 \sin \left( \frac{k\pi}{j} \right)} = \frac j 2 \frac{ \sin \left( 2k\pi - \frac{k \pi}{j} \right)}{\sin\left( \frac{k \pi}{j} \right)} = - \frac j 2.\]
We therefore have
\[  \sum_{p=1} ^{j-1} p j \cos \left( \frac{2k\pi}{j} p\right) = - \frac{j^2}{2}.\]
We also need to compute $\sum_{p=1} ^{j-1} p^2 \cos \left( \frac{2kp \pi}{j} \right)$. 
For this we differentiate \cite[1.352]{gr}, obtaining
\[ \frac{d}{dx} \sum_{k=1} ^{n-1} k \sin(kx) = \sum_{k=1} ^{n-1} k^2 \cos(kx) = \frac{d}{dx} \left( \frac{\sin(nx)}{4 \sin^2 \left( \frac x 2 \right) } - \frac{n \cos \left( \frac{2n-1}{2} x \right)}{2 \sin\left( \frac x 2 \right) } \right)\]
\[ = \frac 1 4 \frac{ n \cos(nx) \sin^2 \left( \frac x 2 \right) - \sin(nx) 2 \sin \left( \frac x 2 \right) \cos \left( \frac x 2 \right) \frac 1 2}{\sin^4 \left( \frac x 2\right)} \]
\[ - \frac n 2 \left[ \frac{(-1) \left( \frac{2n-1}{2} \right) \sin \left( \frac{2n-1}{2} x \right) \sin \left( \frac x 2 \right) -  \cos \left( \frac{2n-1}{2} x \right) \cos \left( \frac x 2 \right) \frac 1 2}{\sin^2 \left( \frac x 2 \right)} \right]\]
\[ = \frac 1 4 \frac{ n \cos(nx) \sin \left( \frac x 2 \right) - \sin(nx) \cos \left( \frac x 2 \right)}{\sin^3 \left( \frac x 2 \right)} + \frac n 2 \frac{\left( \frac{2n-1}{2} \right) \sin \left( \frac{2n-1}{2} x \right) \sin \left( \frac x 2 \right) + \cos \left( \frac{2n-1}{2} x \right) \cos \left( \frac x 2 \right) \frac 1 2}{\sin^2 \left( \frac x 2 \right)}. \]
We have 
\[ x = \frac{2k\pi}{j}, \quad n = j, \implies  \cos(nx) = \cos \left( j \frac{2k\pi}{j} \right) = 1,\]
\[ \cos \left( \frac{2n-1}{2} x \right) = \cos\left( \left( j - \frac 1 2 \right) \frac{2k\pi}{j} \right) = \cos \left( - \frac{k \pi}{j} \right) = \cos \left( \frac{k\pi}{j} \right),\]
\[ \cos \left( \frac x 2 \right) = \cos \left( \frac{k \pi}{j} \right), \] 
\[   \sin(nx) = \sin \left( j \frac{2k\pi}{j} \right) = 0, \, \sin  \left( \frac{2n-1}{2} x \right) = \sin  \left( - \frac{k \pi}{j} \right) = - \sin \left( \frac{k\pi}{j} \right), \,  \sin \left( \frac x 2 \right) = \sin \left( \frac{k \pi}{j} \right).\]
Consequently,
\[ \sum_{p=1} ^{j-1} p^2 \cos \left( \frac{2 k \pi p}{j} \right) = \frac{j}{4 \sin^2 \left( \frac{k \pi}{j} \right)} + \frac j 2 \frac{ \left[ \left( j - \frac 1 2 \right) \left(- \sin^2 \left( \frac{k \pi}{j} \right)\right) + \frac 1 2 \cos^2 \left( \frac{k \pi}{j} \right) \right]}{\sin^2 \left( \frac{k \pi}{j} \right)}\]
\[ = \frac{j}{4 \sin^2 \left( \frac{k \pi}{j} \right)} + \frac{j}{2 \sin^2 \left( \frac{k \pi}{j} \right)} \left( \frac 1 2 \cos^2 \left( \frac{k \pi}{j} \right) - j \sin^2 \left( \frac{k \pi}{j} \right) + \frac 1 2 \sin^2 \left( \frac{k \pi}{j} \right) \right) \]
\[= \frac{j}{4 \sin^2 \left( \frac{k \pi}{j} \right)} + \frac{j}{2 \sin^2 \left( \frac{k \pi}{j} \right)} \left( \frac 1 2 -  j \sin^2 \left( \frac{k \pi}{j} \right) \right)  =  \frac{j}{2 \sin^2 \left( \frac{k \pi}{j} \right)} - \frac 1 2 j^2.\]
We therefore have
\[ \sum_{p=1} ^{j-1} p (j-p) \cos \left( \frac{2kp \pi}{j} \right) =  - \frac{j^2}{2} - \left( \frac{j}{2 \sin^2 \left( \frac{k \pi}{j} \right)} - \frac 1 2 j^2 \right) = - \frac{j}{2 \sin^2 \left( \frac{k \pi}{j} \right)}.  \]
This shows that
\[  \frac{1}{\pi j} \sum_{p=1} ^{j-1}  p (j-p) \cos \left( \frac{2kp \pi}{j} \right) = - \frac{1}{2\pi \sin^2 \left( \frac{k \pi}{j} \right)}. \]
Note that
\[ \left\lfloor \frac{j+1}{2} \right\rfloor - 1 = \left \lfloor \frac{j-1}{2} \right\rfloor \implies 
 \frac{1}{\pi j} \sum_{p=1} ^{j-1}  p (j-p)  \sum_{k=1} ^{\left\lfloor \frac{j+1}{2} \right\rfloor - 1} \left[ \cos \left( \frac{2kp \pi}{j} \right) \log \sin \left(\frac {k \pi }{j} \right) \right]  \]
\[ = \sum_{k=1} ^{ \left \lfloor \frac{j-1}{2} \right\rfloor} \frac{1}{\pi j} \sum_{p=1} ^{j-1}  p (j-p)  \left[ \cos \left( \frac{2kp \pi}{j} \right) \log \sin \left(\frac {k \pi }{j} \right) \right]    =  - \frac{1}{2\pi} \sum_{k=1} ^{ \left \lfloor \frac{j-1}{2} \right\rfloor} \frac{\log \sin \left(\frac {k \pi }{j} \right)}{\sin^2 \left( \frac{k \pi}{j} \right)}.\]
\end{proof}

\begin{proof}[Proof of \eqref{eq:th2_eq2} in Theorem \ref{th:main2}]
For the reader's convenience we recall \eqref{eq:xic_zetac} here:
\[ \xi_c'(0) = \frac 1 2 \left[ \zeta_c '(0) + \frac{5}{24c} - \frac{1}{12} \log 2 \left( c + \frac 1 c \right) \right]. \]
Consequently by Proposition \ref{prop:barnes1}
\begin{eqnarray*}
\left . \frac{d}{dc} \xi_c '(0) \right|_{c=j} &=& \frac 1 2 \left[  - \frac{1}{12} - \frac{1}{8j^2} + \frac{1-j^2}{12 j^2} \log j  - \frac{1}{2 j^3} \sum_{p=1} ^j p(j-p) \psi \left( \frac p j \right)  - \frac{5}{24j^2}  \right.  - \left. \frac{1}{12} \log(2) \left( 1 - \frac{1}{j^2} \right) \right] \\
 &=&  \frac 1 2 \left[  - \frac{1}{12} - \frac{1}{3j^2} + \frac{(1-j^2) \log(2j)}{12j^2} -  \frac{1}{2 j^3} \sum_{p=1} ^j p(j-p) \psi \left( \frac p j \right) \right]. 
 \end{eqnarray*}
Since
\[ \sum_{p=1} ^j p(j-p) = \frac{j (j+1)(j-1)}{6} \implies \frac{1-j^2}{12 j^2} \log(2j) = - \frac{1}{2j^3} \sum_{p=1} ^j p(j-p) \log(2j),\]
Consequently we obtain that
\[ \left . \frac{d}{dc} \xi_c '(0) \right|_{c=j} = \frac 1 2  \left[ - \frac{1}{12} - \frac{1}{3j^2} - \frac{1}{2j^3} \sum_{p=1} ^j p(j-p) \left[ \log(2j) + \psi \left( \frac p j \right) \right] \right]. \]
By Proposition \ref{prop:psi_magic} we therefore have
\[  \left . \frac{d}{dc} \xi_c '(0) \right|_{c=j} = \frac 1 2  \left[ -\frac{1}{12} - \frac{1}{3j^2} + \frac{1}{2j^2} \left[ \frac{\gamma_e (j^2-1)}{6} + \sum_{k=1} ^{\lfloor \frac{j-1}{2} \rfloor} \frac{ \log | \sin(k \pi/j)|}{\sin^2(k \pi /j)} \right] \right]. \]
It is straightforward to show that
\[ \lfloor \frac{j-1}{2} \rfloor = \lceil \frac j 2 - 1 \rceil \implies  \left . \frac{d}{dc} \xi_c '(0) \right|_{c=j} = \frac 1 2  \left[ - \frac{1}{12} - \frac{1}{3j^2} + \frac{1}{2j^2} \left[ \frac{\gamma_e (j^2-1)}{6} + \sum_{k=1} ^{\lceil \frac j 2 - 1 \rceil} \frac{ \log | \sin(k \pi/j)|}{\sin^2(k \pi /j)} \right] \right] \]
\beq = -\frac{1}{24} - \frac{1}{6j^2} + \frac{\gamma_e(j^2-1)}{24j^2} + \frac{1}{4j^2} \sum_{k=1} ^{\lceil \frac j 2 - 1 \rceil} \frac{ \log | \sin(k \pi/j)|}{\sin^2(k \pi /j)}. \label{eq:v1_mag} \eeq
Next, we recall the calculation from Proposition \ref{prop:v2} that was obtained by an entirely independent method
\[  \frac{d}{dc} \xi_c '(0) = - \frac{\pi}{2c^2} \left( \frac{1}{3\pi} + \frac{c^2}{12 \pi} +  \sum_{k = \lceil{-\frac c 2 \rceil}, k \neq 0 } ^{\lceil \frac c 2 - 1 \rceil} \frac{-2\gamma_e + \log 2 - \log\left({1-\cos(2k\pi/c)}\right) }{4 \pi (1-\cos(2k\pi/c))} \right) \]
\[ - \frac{\pi}{2c^2} \left( \frac{2c}{\pi} \sin(\pi c) \int_\R \frac{ - \log 2 + 2 \gamma_e + \log(1+\cosh(s))}{16 \pi (1+\cosh(s)) (\cosh(c s) - \cos(c^2/\pi))} ds \right). \]
Since $c=j$ is an integer, $\sin(\pi c) = 0$, and the expression becomes
\[ - \frac{1}{6j^2} - \frac{1}{24} - \frac{1}{2j^2}  \sum_{k = \lceil{-\frac j 2 \rceil}, k \neq 0 } ^{\lceil \frac j 2 - 1 \rceil} \frac{-2\gamma_e + \log 2 - \log\left({1-\cos(2k\pi/j)}\right) }{4 (1-\cos(2k\pi/j))}. \]
By Proposition \ref{prop:reformulate_sum_w_alpha},
\[  \sum_{k = \lceil{-\frac j 2 \rceil}, k \neq 0 } ^{\lceil \frac j 2 - 1 \rceil} \frac{-2\gamma_e + \log 2 - \log\left({1-\cos(2k\pi/j)}\right) }{4 (1-\cos(2k\pi/j))} = \frac{\gamma_e}{12} (j^2-1) + \frac{1}{2} \sum_{k=1} ^{\lceil \frac j 2 - 1 \rceil} \frac{ \log|\sin(k\pi/j)|}{\sin^2(k\pi/j)}. \]
Consequently, the expression from Proposition \ref{prop:v2} becomes
\[ - \frac{1}{24} - \frac{1}{6j^2} - \frac{\gamma_e}{24j^2} (j^2-1) - \frac{1}{4j^2} \sum_{k=1} ^{\lceil \frac j 2 - 1 \rceil} \frac{ \log|\sin(k\pi/j)|}{\sin^2(k\pi/j)},\]
indeed coinciding with \eqref{eq:v1_mag} that was obtained by a completely independent method.  
\end{proof}


\subsection{The variation of the Bessel zeta function when the parameter is not a natural number} \label{ss:generalangles}  
To complete the proof of Theorem \ref{th:main2} when the parameter is not an integer, we require the calculation of certain integrals.  This is achieved in the following two lemmas.  In these lemmas we prove that a certain integral can be replaced by a sum.  If one were to guess, this should follow from a complex residue calculation, but it is not immediately obvious, because to prove these results, we do not simply apply the residue theorem to the integrand.  Instead, we make a clever choice of an auxiliary function on which to apply the residue theorem.  This technique may provide insights for other contexts in which one would like to transform integrals into finite, explicit sums in fully closed form. 

\begin{lemma} \label{le:residuesnolog}
We have the following identity for $0<c \not \in \N$
\[ \frac{c}{4\pi^2} \sin \left( \pi c \right) \int_\R \frac{ds}{(1+\cosh s)\left( \cosh(cs) - \cos (\pi c) \right) } = \frac{1}{12} \left( \frac 1 \pi - \frac{c^2}{\pi} \right) + \frac{1}{2\pi} \sum_{n=\ceil*{-\frac c 2}, n \neq 0} ^{\floor{ \frac c 2}} \frac{1}{1-\cos(2 \pi n/c)}.\]
\end{lemma}

\begin{proof}
For $s$ real, we define
\[ f(s) := i \sin \left( \pi c\right) \frac{1}{(1+\cosh s) \left( \cosh(cs)- \cos(\pi c) \right) }.\]
Note that
\[ \cosh(cs)- \cos(\pi c)= \frac 1 2 e^{cs} \left( e^{-cs} - e^{i \pi c } \right) \left( e^{-cs} - e^{-i \pi c } \right), \]
from which we obtain
\[ \frac{1}{ \cosh(cs) - \cos(\pi c)} = \frac{1}{i \sin(\pi c)} e^{-cs} \left[ \frac{1}{e^{-cs} - e^{i \pi c}} - \frac{1}{e^{-cs} - e^{-i \pi c}} \right].\]
For $s\in \R$, set
\[ g(s) = \frac{e^{-cs}}{e^{- cs} - e^{i \pi c}} \implies f(s) = \frac{1}{1+\cosh s} \left( g(s) - \overline{ g(s)} \right) =: h(s) - \overline{h(s)}.\]
Now, we extend $h$ to $\C$ with the same definition as above. Further observe that for $s\in \R$ we have $h(-2\pi i + s)$ $= \overline{h(s)}$. 
Consequently, we have the equality that 
\[  \int_{-R} ^R f(s) ds = \int_{-R} ^{R} h(s) ds - \int_{-R} ^R \overline{h(s)} ds = \int_{-R} ^R h(s) ds + \int_{R-2\pi i} ^{-R - 2\pi i} h(z) dz. \]  

Consider the contour integral of $h(s)$ over the clockwise contour from $-R$ to $R$ to $R-2\pi i$ to $-R-2\pi i$ and returning back to $-R$; see Figure \ref{fig:contour1}.  We denote this contour by $\Gamma_R$. For large $R$ it is straightforward to estimate that the integrals along the sides of this box, from $R$ to $R-2\pi i$ and from $-R - 2\pi i$ to $-R$ vanish as $R \to \infty$.  We therefore obtain, owing to the negative orientation of the contour 
\[\int_{-\infty} ^\infty  f(s) ds = \lim_{R \to \infty} \int_{\Gamma_R}  h(s)  ds = - 2\pi i \sum  \Res(h(s); s_n),\]
where the sum is taken over all poles between the two parallel lines $\R$, and $\R - 2\pi i$.  
The function $h(s)$ has a third order pole at $s=-i \pi$.   It has first order poles due to the vanishing of
\[ e^{- cs} - e^{i \pi c}, \quad \textrm{for } s= - i (\pi + 2 n \pi/c), \quad n \in \Z \bigcap \left[ \left( - \frac c 2 , 0 \right)  \cup \left( 0, \frac c 2 \right) \right]. \]
We begin by computing the residues at the simple poles:
\[ \left . \Res(h(s)) \right|_{s=s_n} = \frac{1}{1+\cosh(s_n)} \frac{e^{-c s_n}}{ - c e^{- s_n c} } = - \frac{1}{c (1+\cosh(s_n))}, \]
which for $s_n = - i (\pi + 2 n \pi/c)$ is equal to
\[ - \frac{1}{c (1- \cos(2 n \pi/c)}, \quad n \in \Z \bigcap \left[ \left( - \frac c 2 , 0 \right)  \cup \left( 0, \frac c 2 \right) \right]. \]
Note that the left side of the equality given in the statement of the lemma is then equal to 
\beq 
 - \frac{ic}{4\pi^2} \int_{-\infty} ^\infty f(s) ds 
 = - \frac{c}{2\pi}  \sum \Res(h(s); s_n).  \label{eq:integralresidues} \eeq 
Consequently, the contribution from the simple poles of $h$ is
\[ \frac{1}{2\pi} \sum_{n=\ceil*{-c/2}, n \neq 0} ^{\floor{c/2}} \frac{1}{1-\cos(2 \pi n/c)}.\]

\begin{figure} \centering
\begin{tikzpicture}[decoration={markings,
    mark=at position 1cm   with {\arrow[line width=1pt]{stealth}},
    mark=at position 4.5cm with {\arrow[line width=1pt]{stealth}},
    mark=at position 7cm   with {\arrow[line width=1pt]{stealth}},
    mark=at position 9.5cm with {\arrow[line width=1pt]{stealth}}
  }]
  \draw[thick, ->] (-6,4) -- (6,4) coordinate (xaxis);

  \draw[thick, ->] (0,0) -- (0,6) coordinate (yaxis);

  \node[above] at (xaxis) {$\mathrm{Re}(s)$};

  \node[right]  at (yaxis) {$\mathrm{Im}(s)$};

  \node[right] at (0,1.7) {$-2\pi i$};

  \node[right] at (0, 4.4) {$\Gamma_R$};

  \path[draw,blue, line width=0.8pt, postaction=decorate]
        (-4,4) node[above, black] {$-R$}
    --  (4, 4)   node[above, black] {$R$}
    --  (4, 2)
    --  (-4,2) -- (-4,4);
       \node at (0,2)[circle,fill,inner sep=1.5pt]{};

\end{tikzpicture}
\caption{The contour of integration, $\Gamma_R$.} \label{fig:contour1}
\end{figure}
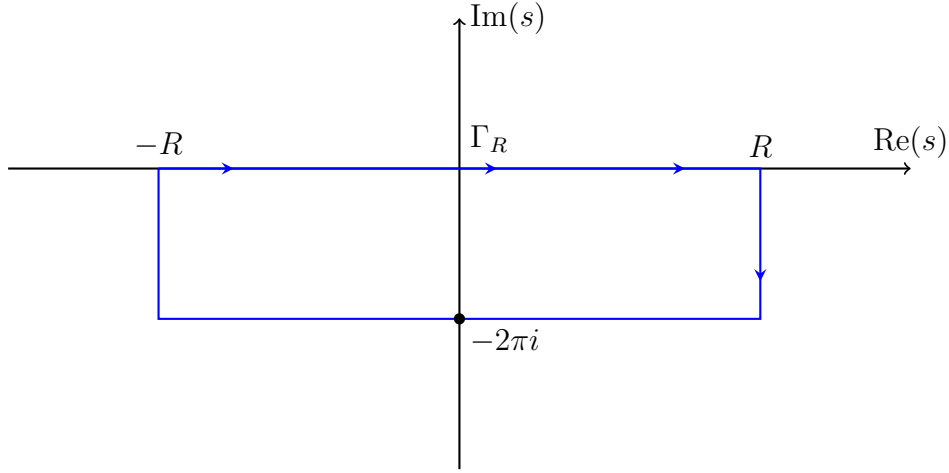

The residue at the third order pole is
\[ \frac 1 2\left .  \frac{d^2}{dz^2} \left[ (z+i \pi)^3 h(z) \right] \right|_{z=-i \pi}.\]
Since we have
\[ h(z) = \frac{e^{-cz}}{(1+\cosh(z)) \left(e^{-cz} - e^{i \pi c }\right)}, \]
we shall compute the series expansions of the functions in these expressions near $z= - i \pi$.  First, we have
\[ 1+ \cosh(z) = \sum_{n \geq 1} - \frac{ (z+i \pi)^{2n}}{(2n)!} = - (z+i\pi)^2 \sum_{n \geq 0} \frac{ (z+i\pi)^{2n}}{(2(n+1))!}.\]
Next we have
\[ e^{-cz} - e^{i \pi c} = \sum_{n \geq 1} \frac{ (z+i\pi)^n}{n!} \left( - c \right)^n e^{i\pi c} = (z+i\pi) e^{i \pi c} \sum_{n \geq 0}  \frac{ (z+i\pi)^n}{(n+1)!} \left( - c \right)^{n+1}.\]
We therefore wish to compute
\[ - \frac{1}{2 e^{i \pi c}} \left . \frac{d^2}{dz^2} \frac{ (z+i\pi)^3 e^{-c z}}{ (z+i\pi)^3 \sum_{n \geq 0} \frac{ (z+i\pi)^{2n}}{(2(n+1))!}  \sum_{m \geq 0}  \frac{ (z+i\pi)^m}{(m+1)!} ( - c )^{m+1} } \right|_{z=-i\pi}\]
\[ = - \frac{1}{2 e^{i \pi c}} \left . \frac{d^2}{dz^2} \frac{e^{-cz}}{  \sum_{n \geq 0} \frac{ (z+i\pi)^{2n}}{(2(n+1))!}  \sum_{m \geq 0}  \frac{ (z+i\pi)^m}{(m+1)!} \left( - c \right)^{m+1} } \right|_{z=-i\pi}.\]
To simplify notations, let us define here
\[ \varphi(z):=  \sum_{n \geq 0} \frac{ (z+i\pi)^{2n}}{(2(n+1))!}, \quad \psi(z) :=  \sum_{n \geq 0}  \frac{ (z+i\pi)^n}{(n+1)!} \left( - c \right)^{n+1} .\]
Then
\[ \frac{d}{dz} \frac{e^{-cz}}{\varphi \psi} = \frac{ - c e^{-cz} \varphi \psi - e^{-cz} (\varphi \psi' + \varphi' \psi)}{(\varphi \psi)^2}\]
\[ = - c \frac{e^{-cz}}{\varphi \psi} -  \frac{e^{-cz} (\varphi \psi' + \varphi' \psi)}{(\varphi \psi)^2} =  - c \frac{e^{-cz}}{\varphi \psi} - \frac{e^{-cz} \psi'}{\varphi \psi^2} - \frac{e^{-cz} \varphi'}{\psi \varphi^2}.\]
Next,
\[ \frac{d^2}{dz^2}  \frac{e^{-cz}}{\varphi \psi} = - c \left( \frac{ -c e^{-cz} \varphi \psi - e^{-cz} (\varphi \psi' + \varphi' \psi)}{(\varphi \psi)^2} \right)\]
\[ - \left( \frac{ \left( - c e^{-cz} \psi' + e^{-cz} \psi'' \right) \varphi \psi^2 - e^{- cz} \psi' \left( \varphi' \psi^2 + 2\varphi \psi \psi' \right)}{(\varphi \psi^2)^2} \right) \]
\[ - \left( \frac{ \left( - c e^{-cz} \varphi' + e^{-cz} \varphi'' \right) \psi \varphi^2 - e^{- cz} \varphi' \left( \psi' \varphi^2 + 2\psi \varphi \varphi' \right)}{(\psi \varphi^2)^2} \right) \]
Since $\varphi'(-i\pi) = 0$, at $z=-i\pi$ we must only evaluate the terms
\[ \left(c \right)^2 \frac{e^{-cz}}{\varphi \psi} + c \frac{e^{-cz} \psi'}{\varphi \psi^2} + c \frac{e^{-cz} \psi'}{\varphi \psi^2} - \frac{e^{-cz} \psi''}{\varphi \psi^2} + 2 \frac{e^{-cz} (\psi')^2}{\varphi \psi^3}  - \frac{e^{-cz} \varphi''}{\psi \varphi^2}.\]
We compute this using
\[ \varphi(-i\pi) = \frac 1 2, \quad \varphi''(-i\pi) = \frac{1}{12}, \quad \psi(-i\pi) = - c, \quad \psi'(-i\pi) = \frac{c^2}{2}, \quad \psi''(-i\pi) = - \frac{c^3}{3 },\]
obtaining
\[ e^{i\pi c} \left( - 2c + 2 c  + \frac{2c}{3} - c + \frac{1}{3c} \right) = e^{i\pi c} \left( - \frac{c}{3} + \frac{1}{3c} \right).\]
Consequently, the residue is
\[ - \frac 1 2 \left( - \frac{c}{3}  + \frac{1}{3c} \right) =  \frac{c}{6}  - \frac{1}{6c}.\]
Recalling the factor $-\frac{c}{2 \pi }$ in \eqref{eq:integralresidues}, the contribution to the integral is therefore
\[ - \frac{c^2}{12 \pi} + \frac{1}{12\pi}.  \]
This completes the proof.
\end{proof}

The next step is to calculate the integral with $\log(1+\cosh(s))$ in the numerator.  Again, this is not a straightforward application of the residue theorem, but rather, it involves the choice of a clever auxiliary function.  For this reason both the result and its proof may be useful.

 \begin{lemma} \label{le:logmanipulates}
 For $0<c$ not an integer, we have
 \[ \frac{c}{8\pi^2} \sin(\pi c) \int_{-\infty} ^\infty  \frac{ \log (1+\cosh s)}{(1+\cosh s)(\cosh(cs) - \cos(\pi c))} ds \]
 \[ = \frac{1}{2\pi} \sum_{n=1} ^{\lfloor \frac c 2 \rfloor} \frac{\log(1-\cos(2n\pi/c))}{1-\cos(2n\pi/c)}  + \frac{c^2}{\pi} \int_1 ^\infty \frac{e^{c t}}{(e^t-1)(1-e^{c t})^2} dt \]
 \[  + \frac{c^2}{\pi}  \int_0 ^1  \left[\frac{e^{c t}}{(e^t-1)(1-e^{ct})^2} - \left(  \frac{1}{c^2 t^3} - \frac{1}{2 c^2 t^2} + \frac{1}{12 t c^2} - \frac{1}{12 t^2} \right)\right] dt  - \frac{c}{4\pi^2} \left( \frac {\pi}{2c} + \frac{\pi c}{3} +\log 2 \left( \frac{\pi}{6 c} - \frac{\pi c}{6} \right) \right).\]
\end{lemma}

\begin{proof}
 To demonstrate the lemma, we compute
 \[ I(c) := i \sin \left( \pi c \right) \int_{-\infty} ^\infty \frac{ \log (1+\cosh s)}{(1+\cosh s)(\cosh(sc) - \cos(\pi c))}  ds.\]
 Similar to our previous calculations, we re-write the expression
 \[ I(c) = \int_{-\infty} ^\infty  \frac{ \log(1+\cosh s)}{(1+\cosh s)} e^{-cs} \left[ \frac{1}{e^{-cs} - e^{i \pi c}}-\frac{1}{e^{-cs} - e^{-i \pi c}} \right] ds. \]
We therefore define 
\[ h(s) = \frac{\log(1+\cosh s)}{1+\cosh s} e^{-sc} \frac{1}{e^{-sc} - e^{i \pi c}}.\]
Then for $s$ real we have
\[ I(c) = \int_\R [ h(s) - \overline{h(s)}] ds = 2 i \Ima \int_{-\infty} ^\infty h(s) ds. \]
We consider the contour integral as depicted in Figure \ref{fig:contourwow}.  The integrand has a branch point of the logarithm at $s=-i\pi$, hence the contour's avoidance of that point.  The integrand has simple poles at points $s_n$ with
\[ s_n = -i \pi + 2 i n \pi/c, \]
such that these points lie within the contour, which requires
\[ -\pi < -\pi + 2n\pi/c < 0 \implies 1 \leq n < \frac c 2.\]

By the residue theorem, since the contour is negatively oriented,
\[ \int_{\Gamma_{R, \epsilon}} h(s) ds = - 2\pi i \sum_{n=1} ^{\lfloor \frac c2 \rfloor}  \Res(h; -i\pi + 2 in \pi/c).\]
The residues at these simple poles are computed as above, and are equal to
\[ - \frac { \log(1-\cos(2n\pi/c))}{c (1-\cos(2n\pi/c))}.\]
We therefore obtain that
\[  \int_{\Gamma_{R, \epsilon}} h(s) ds = 2 i \sum_{n=1} ^{\lfloor \frac c2 \rfloor} \frac{\pi  \log(1-\cos(2n\pi/c))}{c(1-\cos(2n\pi/c))}.\]
On the other hand
\[ I(c) = 2 i \Ima \left( 2 i \sum_{n=1} ^{\lfloor \frac c 2 \rfloor}  \frac{\pi \log(1-\cos(2n\pi/c))}{c(1-\cos(2n\pi/c))} - \lim_{R \to \infty} \int_{\Gamma_- \cup \Gamma_\eps \cup \Gamma_+} h(s) ds \right)\]
\[ = 4 i \sum_{n=1} ^{\lfloor \frac c 2 \rfloor}  \frac{\pi  \log(1-\cos(2n\pi/c))}{c(1-\cos(2n\pi/c))} - 2 i \Ima \left( \lim_{R \to \infty} \int_{\Gamma_- \cup \Gamma_\eps \cup \Gamma_+} h(s) ds \right).\]

\begin{figure} \centering
\begin{tikzpicture}[decoration={markings,
    mark=at position 1cm   with {\arrow[line width=1pt]{stealth}},
    mark=at position 4.5cm with {\arrow[line width=1pt]{stealth}},
    mark=at position 7cm   with {\arrow[line width=1pt]{stealth}},
    mark=at position 9.5cm with {\arrow[line width=1pt]{stealth}}
  }]
  \draw[thick, ->] (-6,4) -- (6,4) coordinate (xaxis);

  \draw[thick, ->] (0,0) -- (0,6) coordinate (yaxis);

  \node[above] at (xaxis) {$\mathrm{Re}(s)$};

  \node[right]  at (yaxis) {$\mathrm{Im}(s)$};

  \node[right] at (0,1.7) {$-i\pi$};

  \node[left] at (-2, 1.7) {$\Gamma_-$};

    \node[right] at (2, 1.7) {$\Gamma_+$};

  \path[draw,blue, line width=0.8pt, postaction=decorate]
        (-4,4) node[above, black] {$-R$}
    --  (4, 4)   node[above, black] {$R$}
    --  (4, 2)
    --  (0.5,2);

      \path[draw,blue, line width=0.8pt, postaction=decorate]
        (-0.5,2) -- (-4,2) --  (-4,4);
        \coordinate (A) at (0.5,2); \coordinate (B) at (-0.5,2); \coordinate (C) at (0,2);

       \draw pic[draw, -, blue, line width=0.8pt, postaction=decorate, angle radius =0.5cm]{angle=A--C--B};
       \node at (0,2)[circle,fill,inner sep=1.5pt]{};
       \node[right] at (0,2.7) {$\Gamma_\epsilon$};

\end{tikzpicture}
\caption{The contour of integration, $\Gamma_{R, \epsilon}$.} \label{fig:contourwow}
\end{figure}
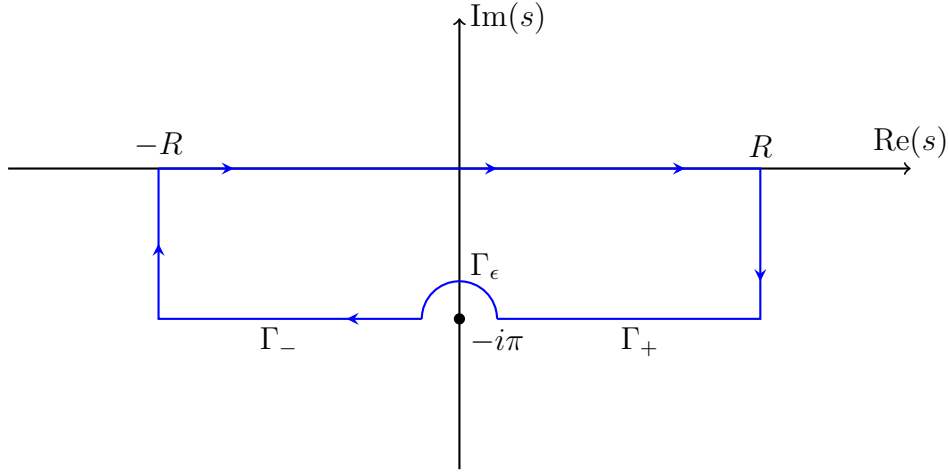

If we parametrize the integral over $\Gamma_-$ by $s=-i\pi + t$ for $t \in [-R, -\eps]$ then we have that
\[ h(s) = \frac{\log(1-\cosh t) e^{-c t} e^{i \pi c}} {(1-\cosh t)(e^{-ct} e^{i\pi c} - e^{i\pi c})}
= \frac{ \log(1-\cosh t) e^{-c t}}{(1-\cosh t)(e^{-ct} - 1)}.\]
Since $t \in [-R, -\eps]$ with $\eps > 0$, $\cosh t > 1$, and therefore the argument of the complex logarithm is $\pi$.  Consequently,
\[ s \in \Gamma_- \implies \Ima h(s) = \frac{\pi  e^{-c t}}{(1-\cosh t)(e^{-ct} - 1)}.\]
We therefore have, recalling the orientation
\[ \lim_{R \to \infty} \Ima \int_{\Gamma_-} h(s) ds = \pi \int_{-\eps} ^{-\infty} \frac{ e^{-ct}}{(1-\cosh t)(e^{-ct} - 1)} dt = - \pi \int_{\eps} ^\infty \frac{ e^{c t}}{(1-\cosh t)(e^{c t} - 1)} dt.\]
For $\Gamma_+$ the argument of the complex logarithm should be taken to equal $-\pi$, because this is on the opposite side of the branch cut for the logarithm.
Similarly, we have
\[ \lim_{R \to \infty} \Ima \int_{\Gamma_+} h(s) ds = - \pi \int_{\infty} ^\eps \frac{ e^{-c t}}{(1-\cosh t)(e^{-c t} - 1)} dt = \pi \int_\eps ^\infty  \frac{ e^{-c t}}{(1-\cosh t)(e^{-ct} - 1)} dt. \]
Consequently, we obtain that
\[ \lim_{R \to \infty} \Ima \left( \int_{\Gamma_-} h(s) ds + \int_{\Gamma_+} h(s) ds \right) =  \pi \int_{\eps} ^\infty \frac{1+e^{ct}}{(1-\cosh t)(1-e^{c t})} dt. \]

Note that
\[ \frac{1}{1-\cosh t} = - \frac{2 e^t}{(1-e^t)^2} = -2 \frac{d}{dt} \frac{1}{1-e^t}.\]
This is useful for partial integration
\[  \pi \int_{\eps} ^\infty \frac{1+e^{ct}}{(1-\cosh t)(1-e^{ct} )} dt = -2\pi \int_\eps ^\infty  \frac{1+e^{ct}}{(1-e^{ct})} \frac{d}{dt} \frac{1}{1-e^t}  dt \]
\[ = - 2 \pi \left\{ \left . \frac{1+e^{ct}}{1-e^{ct}} \frac{1}{1-e^t}\right|_{\eps} ^\infty - \int_\eps ^\infty  \frac{1}{1-e^t} \frac{d}{dt} \frac{1+e^{c t}}{1-e^{ct}} dt \right\}  = 2 \pi \frac{1+e^{c\eps}}{1-e^{c\eps}} \frac{1}{1-e^\eps} - 4\pi c \int_\eps ^\infty  \frac{e^{ct}}{(e^t-1)(1-e^{c t})^2} dt .\]
We would like to send $\eps \searrow 0$ and combine with the integral around $\Gamma_\eps$.  For this reason we add and subtract the small $t$ asymptotic behavior.  Consequently we observe that
\[ \frac{e^{ct}}{(e^t - 1)(e^{ct} - 1)^2} \approx \frac{1}{c^2 t^3} - \frac{1}{2 c^2 t^2} +  \frac{1}{12t} \left( \frac{1}{c^2} - 1 \right) + \cO(t^0). \]
The adding-subtracting game is only required in the integral from $\eps$ to $1$.  Consequently, we have
\[ \lim_{R \to \infty} \Ima \left( \int_{\Gamma_-} h(s) ds + \int_{\Gamma_+} h(s) ds \right) = 2 \pi \frac{1+e^{c \eps}}{1-e^{c\eps}} \frac{1}{1-e^\eps} - 4\pi c  \int_1 ^\infty  \frac{e^{c t }}{(e^t-1)(1-e^{c t})^2}dt \]
\[ - 4\pi c \int_\eps ^1  \left[\frac{e^{ct }}{(e^t-1)(1-e^{ct })^2} - \left( \frac{1}{c^2 t^3} - \frac{1}{2 c^2 t^2} +  \frac{1}{12t} \left( \frac{1}{c^2} - 1 \right) \right) \right] dt \]
\[ - 4 \pi c \int_\eps ^1  \left( \frac{1}{c^2 t^3} - \frac{1}{2 c^2 t^2} +  \frac{1}{12t} \left( \frac{1}{c^2} - 1 \right) \right) dt.\]
The last term we evaluate directly
\begin{multline*} - 4\pi c \left . \left( - \frac{1}{2 c^2 t^2} + \frac{1}{2 c^2 t} + \frac{1}{12} \left( \frac{1}{c^2} - 1 \right)   \log t \right) \right|_{\eps} ^1 \\
 = - 4 \pi c \left( - \frac{1}{2 c^2} + \frac{1}{2c^2} + \frac{1}{2c^2 \eps^2} - \frac{1}{2 c^2 \eps } - \frac{1}{12} \left( \frac{1}{c^2} - 1 \right) \log \eps \right) \\ 
 = - 4\pi c  \left( \frac{1}{2 c^2 \eps^2} - \frac{1}{2 c^2 \eps } -  \frac{1}{12} \left( \frac{1}{c^2} - 1 \right)  \log \eps\right) = - \frac{2\pi}{c\eps^2} + \frac{2 \pi }{c\eps} + \frac{\pi c}{3} \left( \frac{1}{c^2} - 1 \right) \log \eps.
 \end{multline*}
We therefore have computed  
\[ \lim_{R \to \infty} \Ima \left( \int_{\Gamma_-} h(s) ds + \int_{\Gamma_+} h(s) ds \right) = - 4\pi c \int_1 ^\infty  \frac{e^{c t}}{(e^t-1)(1-e^{ct})^2}\ dt \]
\[ - 4\pi c \int_\eps ^1  \left[\frac{e^{ct}}{(e^t-1)(1-e^{ct})^2} - \left(  \frac{1}{c^2 t^3} - \frac{1}{2 c^2 t^2} + \frac{1}{12t} \left( \frac{1}{c^2} - 1 \right)  \right)\right] \ dt\] 
\[ +  2 \pi \frac{1+e^{c \eps}}{1-e^{c \eps}} \frac{1}{1-e^\eps} - \frac{2\pi}{c\eps^2} + \frac{2 \pi }{c\eps} + \frac{\pi c}{3} \left( \frac{1}{c^2} - 1 \right) \log \eps.\]
We compute the Laurent series expansion near $\eps = 0$
\[ 2 \pi \frac{1+e^{c \eps}}{1-e^{c \eps}} \frac{1}{1-e^\eps} \sim \frac{4\pi}{c\eps^2} - \frac{2\pi}{c\eps} + \frac{\pi}{3c} + \frac{\pi c}{3} + \cO (\eps). \]
Consequently
\[ \lim_{R \to \infty} \Ima \left( \int_{\Gamma_-} h(s) ds + \int_{\Gamma_+} h(s) ds \right) = - 4\pi c  \int_1 ^\infty  \frac{e^{c t}}{(e^t-1)(1-e^{c t})^2} dt \]
\[ - 4\pi c \int_\eps ^1  \left[\frac{e^{c t}}{(e^t-1)(1-e^{ct})^2} - \left(    \frac{1}{c^2 t^3} - \frac{1}{2 c^2 t^2} + \frac{1}{12t} \left( \frac{1}{c^2} - 1 \right) \right) \right] dt \]
\[ + \frac{2\pi }{c\eps^2} + \frac{\pi}{3c} + \frac{\pi c}{3} + \left( \frac{\pi}{3c} - \frac{\pi c}{3} \right) \log \eps + \cO(\eps). \]

Finally, it remains to compute the integral over $\Gamma_\eps$.  For this calculation, we require the expansion of $h(s)$ for $s=-i\pi + z$ for $|z|$ small.  Note that
\[ 1+\cosh s = - \frac 1 2 z^2 \left( 1 + \frac{z^2}{12} + \mathcal O (z^4) \right), \quad s = - i\pi + z, \quad |z| \textrm{ small,}\]
and therefore
\[ \log(1+\cosh s) = \log \left( - \frac 1 2 z^2 \right) + \frac{1}{12} z^2 + \mathcal O (z^4), \quad |z| \textrm{ small.}\]
Along $\Gamma_\eps$, $z = \eps e^{i\theta}$, for $\theta \in [0,\pi]$.  Consequently,
\[ \log (1+\cosh s) = \log \left( \frac{\eps^2}{2} \right) + i (2\theta - \pi) + \frac{1}{12} \eps^2 e^{2i\theta} + \mathcal O (\eps^4).\]
We further use the expansion
\[ \frac{e^{- c s }}{(1+\cosh s)(e^{-c s} - e^{i \pi c})} = \frac{2}{c z^3} - \frac{1}{z^2} + \frac1 z \frac 1 6 \left( c - \frac 1 c \right) + \mathcal O (z^0).\]
With the parametrization $s=-i\pi + \eps e^{i\theta}$ we have that $ds = i \eps e^{i\theta} d\theta$.  Consequently we require the terms in the expansion up to order $\eps^{-1}$.  We have
\[ \frac{ \log(1+\cosh s) e^{-c s }}{(1+\cosh s)(e^{-c s} - e^{i\pi c} )} = \log \left( \frac{\eps^2}{2} \right) \left( \frac{2}{c \eps^3 e^{3i\theta}} - \frac{1}{\eps^2 e^{2 i \theta}} + \frac{1}{6 \eps e^{i\theta}} \left( c - \frac 1 c \right) \right) \]
\[ + i (2\theta - \pi ) \left( \frac{2}{c \eps^3 e^{3i\theta}} - \frac{1}{\eps^2 e^{2 i \theta}} + \frac{1}{6 \eps e^{i\theta}} \left( c - \frac 1 c \right) \right) + \frac{2}{12 c \eps e^{i\theta}} + \mathcal O (\log \eps).\]
We shall now compute the imaginary parts of the relevant integrals along $\Gamma_\eps$:
\[\Ima \int_0 ^\pi  \log \left( \frac{\eps^2}{2} \right) \left( \frac{2}{c \eps^3 e^{3i\theta}} - \frac{1}{\eps^2 e^{2 i \theta}} +
\frac{1}{6 \eps e^{i\theta}} \left( c - \frac 1 c \right) \right) i \eps e^{i\theta} d\theta \]
\[ = \Ima \int_0 ^\pi i \log \left( \frac{\eps^2}{2} \right) \left( \frac{2}{c \eps^2} e^{-2i\theta} - \frac 1 \eps e^{-i\theta} + \frac 1 6 \left(c- \frac 1 c \right) \right) d\theta \] 
\[= \int_0 ^\pi \log \left( \frac{\eps^2}{2} \right) \left( \frac{2}{c \eps^2} \cos(2\theta) - \frac 1 \eps \cos(\theta) + \frac 1 6 \left( c - \frac 1 c \right)  \right) d\theta = \pi \log \left( \frac{\eps^2}{2} \right)  \frac 1 6 \left( c - \frac 1 c \right),\]
\[ \Ima \int_0 ^\pi  i (2\theta - \pi ) \left( \frac{2}{c \eps^3 e^{3i\theta}} - \frac{1}{\eps^2 e^{2 i \theta}} + \frac{1}{6 \eps e^{i\theta}} \left( c - \frac 1 c \right) \right) i \eps e^{i \theta} d\theta \]
\[ = \Ima \int_0 ^\pi (\pi - 2\theta) \left( \frac{2}{c \eps^2} e^{-2i\theta} - \frac 1 \eps e^{-i\theta} + \frac 1 6 \left(c - \frac 1 c \right) \right) d\theta \]
\[ = \int_0 ^\pi (\pi - 2\theta)  \left( \frac{2}{c \eps^2} \sin(-2\theta) -  \frac 1 \eps \sin(-\theta) \right) d\theta = - \frac{2\pi}{c\eps^2},\]
\[ \Ima \int_0 ^\pi \frac{2}{12 c \eps e^{i\theta}} i \eps e^{i\theta} d\theta = \frac {\pi}{6c},\]
\[  \int_0 ^\pi i \eps e^{i\theta} \mathcal O (\log \eps) d\theta = \mathcal O (\eps \log \eps), \quad \eps \searrow 0.\]
In summary, we have computed
\[ \Ima \int_{\Gamma_\eps} h(s) ds = \pi \log \left( \frac{\eps^2}{2} \right)  \frac 1 6 \left( c - \frac 1 c \right) - \frac{2\pi}{c\eps^2} + \frac {\pi}{6c} + \mathcal O (\eps \log \eps) \]
\[= \log (\eps) \left( \frac{\pi c}{3} - \frac{\pi}{3c} \right) - \pi \log 2 \frac 1 6 \left( c - \frac 1 c \right) - \frac{2\pi}{c\eps^2} + \frac{\pi}{6c} + \mathcal O (\eps \log \eps).\]
We combine this with the calculation of the integrals along $\Gamma_\pm$,
\[ \lim_{R \to \infty} \Ima \left( \int_{\Gamma_-} h(s) ds + \int_{\Gamma_+} h(s) ds \right) + \Ima \int_{\Gamma_\eps} h(s) ds = - 4\pi c \int_1 ^\infty  \frac{e^{c t }}{(e^t-1)(1-e^{c t})^2} dt \]
\[ - 4\pi c \int_\eps ^1  \left[\frac{e^{c t}}{(e^t-1)(1-e^{ct})^2} - \left(  \frac{1}{c^2 t^3} - \frac{1}{2 c^2 t^2} + \frac{1}{12 t} \left( \frac{1}{c^2} - 1 \right) \right)\right]dt  \]
\[ + \frac{2\pi}{c\eps^2} + \frac{\pi}{3c} + \frac{\pi c}{3} + \left( \frac{\pi}{3c} - \frac{\pi c}{3} \right) \log \eps \]

\[ + \log (\eps) \left( \frac{\pi}{3c} - \frac{\pi c}{3} \right)   - \pi \log 2 \frac 1 6 \left( c - \frac 1 c \right) - \frac{2\pi}{c \eps^2} + \frac{\pi}{6c} + \mathcal O (\eps \log \eps),\]
resulting in
\[ - 4\pi c \int_1 ^\infty  \frac{e^{ct}}{(e^t-1)(1-e^{ct})^2} dt  - 4\pi c \int_0 ^1  \left[\frac{e^{ct}}{(e^t-1)(1-e^{ct})^2} - \left(  \frac{1}{c^2 t^3} - \frac{1}{2 c^2 t^2} + \frac{1}{12 t} \left( \frac{1}{c^2} - 1 \right)\right)\right] dt \]
\[ + \frac {\pi}{ 2c} + \frac{\pi c}{3} +\log 2 \left( \frac{\pi}{6c} - \frac{\pi c}{6} \right)  \textrm{ as } \eps \to 0.\]

We have therefore computed
\[ I(c) = i \sin \left( \pi c \right) \int_{-\infty} ^\infty  \frac{ \log (1+\cosh s)}{(1+\cosh s)(\cosh(cs) - \cos(\pi c))} ds \] 
\[= 4 i \sum_{n=1} ^{\lfloor \frac{c}{2} \rfloor}  \frac{\alpha  \log(1-\cos(2n\pi/c))}{1-\cos(2n\pi/c)}  - 2 i \left(  - 4\pi c \int_1 ^\infty  \frac{e^{ct }}{(e^t-1)(1-e^{c t })^2} dt \right) \] 
\[ - 2 i \left( - 4\pi c \int_0 ^1  \left[\frac{e^{c t}}{(e^t-1)(1-e^{ct})^2} - \left(  \frac{1}{c^2 t^3} - \frac{1}{2 c^2 t^2} + \frac{1}{12t} \left( \frac{1}{c^2} - 1 \right)\right)\right] dt \right)  - 2 i \left( \frac {\pi}{2c} + \frac{\pi c}{3} +\log 2 \left(\frac{\pi}{6c} - \frac{\pi c}{6} \right)  \right).\]
Consequently, we obtain that
\[ \sin(\pi c) \int_{-\infty} ^\infty  \frac{ \log (1+\cosh s)}{(1+\cosh s)(\cosh(c s) - \cos(\pi c ))} ds = 4 \frac \pi c \sum_{n=1} ^{\lfloor \frac{c}{2} \rfloor}  \frac{\log(1-\cos(2n\pi/c))}{1-\cos(2n\pi/c)}  + 8 \pi c \int_1 ^\infty  \frac{e^{ct}}{(e^t-1)(1-e^{c t})^2} dt \]
\[ + 8 \pi c \int_0 ^1  \left[\frac{e^{ct}}{(e^t-1)(1-e^{ct})^2} - \left(  \frac{1}{c^2 t^3} - \frac{1}{2 c^2 t^2} + \frac{1}{12 t} \left( \frac{1}{c^2} - 1 \right)\right)\right] dt  - 2\left( \frac {\pi}{2c} + \frac{\pi c}{3} +\log 2 \left( \frac{\pi}{6c} - \frac{\pi c}{6} \right) \right),\]
and therefore
\[ \frac{c}{8\pi^2} \sin(\pi c) \int_{-\infty} ^\infty  \frac{ \log (1+\cosh s)}{(1+\cosh s)(\cosh(cs) - \cos(\pi c))} ds = \frac{1}{2\pi}\sum_{n=1} ^{\lfloor \frac{c}{2} \rfloor}  \frac{\log(1-\cos(2n\pi/c))}{1-\cos(2n\pi/c)} \]
\[ + \frac{c^2}{\pi} \int_1 ^\infty  \frac{e^{ct}}{(e^t-1)(1-e^{ct})^2}dt + \frac{c^2}{\pi}  \int_0 ^1  \left[\frac{e^{ct}}{(e^t-1)(1-e^{ct})^2} - \left(  \frac{1}{c^2 t^3} - \frac{1}{2 c^2 t^2} + \frac{1}{12 t c^2} - \frac{1}{12 t}\right)\right]  dt\]
\[ - \frac{c}{4\pi^2} \left( \frac {\pi}{ 2 c} + \frac{\pi c}{3} +\log 2  \left( \frac{\pi}{6c} - \frac{\pi c}{6} \right) \right).\]
\end{proof}

Now, we can complete the proof of Theorem \ref{th:main2} in the case when $c$ is not a natural number. 
\begin{proof}[Proof of \eqref{eq:th2_eq2} in Theorem \ref{th:main2}]
Proposition \ref{prop:v2} gives the variation as
\[  \frac{d}{dc} \xi_c '(0) = - \frac{\pi}{2c^2} \left( \frac{1}{3\pi} + \frac{c^2}{12 \pi} +  \sum_{k = \lceil{-\frac c 2 \rceil}, k \neq 0 } ^{\lceil \frac c 2 - 1 \rceil} \frac{-2\gamma_e + \log 2 - \log\left({1-\cos(2k\pi/c)}\right) }{4 \pi (1-\cos(2k\pi/c))} \right) \]
\[ - \frac{\pi}{2c^2} \left( \frac{2c}{\pi} \sin(\pi c) \int_\R \frac{ - \log 2 + 2 \gamma_e + \log(1+\cosh(s))}{16 \pi (1+\cosh(s)) (\cosh(c s) - \cos(c \pi))} ds \right) \]
\[ = - \frac{1}{6c^2} - \frac{1}{24} + \frac{1}{8c^2} \sum_{k = \lceil{-\frac c 2 \rceil}, k \neq 0 } ^{\lceil \frac c 2 - 1 \rceil} \frac{2\gamma_e - \log 2 + \log\left(1-\cos(2k\pi/c)\right) }{4 \pi (1-\cos(2k\pi/c))} \] 
\[ + \frac{1}{16 \pi c} \sin(\pi c) \int_\R \frac{\log 2 - 2 \gamma_e - \log(1+\cosh(s))}{(1+\cosh(s))(\cosh(cs)-\cos(\pi c))} ds. \]
By Lemma \ref{le:residuesnolog},
\beq  & \sin(\pi c) \frac{\log 2 - 2 \gamma_e}{16 \pi c} \int_\R \frac{ds}{(1+\cosh(s))(\cosh(cs) - \cos(\pi c))}   = \frac{4\pi^2}{c} \frac{\log 2 - 2 \gamma_e}{16 \pi c} \frac{1}{12} \left( \frac 1 \pi - \frac{c^2}{\pi} \right) \nn \\ 
& +  \frac{4\pi^2}{c} \frac{\log 2 - 2 \gamma_e}{16 \pi c} \frac{1}{2\pi} \sum_{n= \lceil - \frac c 2 \rceil, n \neq 0} ^{\lfloor \frac c 2 \rfloor} \frac{1}{1-\cos(2n\pi/c)} \nn \\ 
& =  \frac{ \log 2 - 2 \gamma_e}{48c^2} \left(1-c^2 \right) + \frac{ \log 2 - 2 \gamma_e}{8c^2} \sum_{n= \lceil - \frac c 2 \rceil, n \neq 0} ^{\lfloor \frac c 2 \rfloor} \frac{1}{1-\cos(2n\pi/c)}. \label{eq:sinpic_term} \eeq 
Since cosine is an even function, and $ \lfloor \frac c 2 \rfloor = \lceil \frac c 2 - 1 \rceil $
\eqref{eq:sinpic_term}  becomes
\beq \frac{ \log 2 - 2 \gamma_e}{48c^2} \left(1-c^2 \right) + \frac{ \log 2 - 2 \gamma_e}{4c^2} \sum_{n=1} ^{\lceil \frac c 2 - 1 \rceil} \frac{1}{1-\cos(2n\pi/c)}. \label{eq:sinpic_2} \eeq 
Moreover, we have the trigonometric identity $ 1 - \cos(2 n \pi/c) = 2 \sin^2(n \pi/c)$, 
so \eqref{eq:sinpic_2} is equal to 
\beq \frac{ \log 2 - 2 \gamma_e}{48c^2} \left(1-c^2 \right) + \frac{ \log 2 - 2 \gamma_e}{8c^2} \sum_{k=1} ^{\lceil \frac c 2 - 1 \rceil} \frac{1}{\sin^2(k \pi/c)}. \label{eq:sinpic_3} \eeq 

By Proposition \ref{prop:reformulate_sum_w_alpha},
\[  \frac{1}{8c^2} \sum_{k = \lceil{-\frac c 2 \rceil}, k \neq 0 } ^{\lceil \frac c 2 - 1 \rceil} \frac{2\gamma_e - \log 2 + \log\left(1-\cos(2k\pi/c)\right) }{4 \pi (1-\cos(2k\pi/c))} =  \frac{1}{4c^2} \sum_{k=1} ^{\lceil \frac c 2 - 1 \rceil} \frac{\gamma_e + \log|\sin(k\pi/c)|}{\sin^2(k\pi/c)}. \]
Consequently, the variational formula in Proposition \ref{prop:v2} is
\[ - \frac{1}{6c^2} - \frac{1}{24} + \frac{1}{4c^2} \sum_{k=1} ^{\lceil \frac c 2 - 1 \rceil} \frac{\gamma_e + \log|\sin(k\pi/c)|}{\sin^2(k\pi/c)} + \frac{ \log 2 - 2 \gamma_e}{48c^2} \left(1-c^2 \right) + \frac{ \log 2 - 2 \gamma_e}{8c^2} \sum_{k=1} ^{\lceil \frac c 2 - 1 \rceil} \frac{1}{\sin^2(k \pi/c)} \]
\[  - \frac{\sin(\pi c)}{16\pi c} \int_\R \frac{ \log(1+\cosh(s))}{(1+\cosh(s))(\cosh(cs)-\cos(\pi c))} ds  =   - \frac{1}{6c^2} - \frac{1}{24} + \frac{1}{4c^2} \sum_{k=1} ^{\lceil \frac c 2 - 1 \rceil} \frac{ \log|\sin(k\pi/c)|}{\sin^2(k\pi/c)} + \frac{\log 2 - 2 \gamma_e}{48c^2} (1-c^2) \]
\[ + \frac{\log 2}{8 c^2} \sum_{k=1} ^{\lceil \frac c 2 - 1 \rceil} \frac{1}{\sin^2(k\pi/c)}  - \frac{\sin(\pi c)}{16\pi c} \int_\R \frac{ \log(1+\cosh(s))}{(1+\cosh(s))(\cosh(cs)-\cos(\pi c))} ds.\]
The last term is by Lemma \ref{le:logmanipulates}
 \[ -  \frac{\pi}{2c^2}   \frac{c}{8\pi^2} \sin(\pi c) \int_{-\infty} ^\infty  \frac{ \log (1+\cosh s)}{(1+\cosh s)(\cosh(cs) - \cos(\pi c))} ds \]
 \[ = - \frac{\pi}{2c^2} \frac{1}{2\pi}  \sum_{n=1} ^{\lfloor \frac c 2 \rfloor}  \frac{\log(1-\cos(2n\pi/c))}{1-\cos(2n\pi/c)}  - \frac{\pi}{2c^2} \frac{c^2}{\pi} \int_1 ^\infty \frac{e^{c t}}{(e^t-1)(1-e^{c t})^2} dt \]
 \[  - \frac{\pi}{2c^2} \frac{c^2}{\pi}  \int_0 ^1  \left[\frac{e^{c t}}{(e^t-1)(1-e^{ct})^2} - \left(  \frac{1}{c^2 t^3} - \frac{1}{2 c^2 t^2} + \frac{1}{12 t c^2} - \frac{1}{12 t} \right)\right] dt \]
\[ + \frac{\pi}{2c^2} \left[ \frac{c}{4\pi^2} \left( \frac {\pi}{2c} + \frac{\pi c}{3} +\log 2 \left( \frac{\pi}{6 c} - \frac{\pi c}{6} \right) \right) \right]\]
\[ = - \frac{1}{4c^2} \sum_{n=1} ^{ \lceil \frac c 2 - 1 \rceil} \frac{ \log(2 \sin^2 (n \pi /c))}{2 \sin^2(n \pi /c)} - \frac 1 2 \int_1 ^\infty \frac{e^{ct}}{(e^t -1)(1-e^{ct})^2} dt \]
\[ - \frac 1 2 \int_0 ^1  \left[\frac{e^{c t}}{(e^t-1)(1-e^{ct})^2} - \left(  \frac{1}{c^2 t^3} - \frac{1}{2 c^2 t^2} + \frac{1}{12 t c^2} - \frac{1}{12 t} \right)\right] dt \]
\[ + \frac{1}{8\pi c} \left[ \frac{\pi}{2c} + \frac{\pi c}{3} + \log 2 \left( \frac{\pi}{6c} - \frac{\pi c}{6} \right) \right] \]
\[ = - \frac{1}{8c^2} \sum_{n=1} ^{ \lceil \frac c 2 - 1 \rceil} \frac{ \log2 + 2 \log| \sin (n \pi /c))|}{ \sin^2(n \pi /c)}  - \frac 1 2 \int_1 ^\infty \frac{e^{ct}}{(e^t -1)(1-e^{ct})^2} dt \]
\[ - \frac 1 2 \int_0 ^1  \left[\frac{e^{c t}}{(e^t-1)(1-e^{ct})^2} - \left(  \frac{1}{c^2 t^3} - \frac{1}{2 c^2 t^2} + \frac{1}{12 t c^2} - \frac{1}{12 t} \right)\right] dt \]
\[ + \frac{1}{16c^2} + \frac{1}{24} + \log 2 \left( \frac{1}{48c^2} - \frac{1}{48} \right). \]
We therefore obtain in total
\[  - \frac{1}{6c^2} - \frac{1}{24} + \frac{1}{4c^2} \sum_{k=1} ^{\lceil \frac c 2 - 1 \rceil} \frac{ \log|\sin(k\pi/c)|}{\sin^2(k\pi/c)} + \frac{\log 2 - 2 \gamma_e}{48c^2} (1-c^2)  + \frac{\log 2}{8 c^2} \sum_{k=1} ^{\lceil \frac c 2 - 1 \rceil} \frac{1}{\sin^2(k\pi/c)} \]
\[- \frac{1}{8c^2} \sum_{n=1} ^{ \lceil \frac c 2 - 1 \rceil} \frac{ \log2 + 2 \log| \sin (n \pi /c))|}{ \sin^2(n \pi /c)}  - \frac 1 2 \int_1 ^\infty \frac{e^{ct}}{(e^t -1)(1-e^{ct})^2} dt \]
\[ - \frac 1 2 \int_0 ^1  \left[\frac{e^{c t}}{(e^t-1)(1-e^{ct})^2} - \left(  \frac{1}{c^2 t^3} - \frac{1}{2 c^2 t^2} + \frac{1}{12 t c^2} - \frac{1}{12 t} \right)\right] dt \]
\[ + \frac{1}{16c^2} + \frac{1}{24} + \log 2 \left( \frac{1}{48c^2} - \frac{1}{48} \right) \]
\[ = - \frac{1}{6c^2} + \frac{1}{16c^2} + \frac{\log 2}{48c^2} (1-c^2) - \frac{\gamma_e}{24c^2} (1-c^2) + \frac{\log 2}{48c^2} (1-c^2)\]
\[ - \frac 1 2 \int_1 ^\infty \frac{e^{ct}}{(e^t -1)(1-e^{ct})^2} dt \]
\[ - \frac 1 2 \int_0 ^1  \left[\frac{e^{c t}}{(e^t-1)(1-e^{ct})^2} - \left(  \frac{1}{c^2 t^3} - \frac{1}{2 c^2 t^2} + \frac{1}{12 t c^2} - \frac{1}{12 t} \right)\right] dt \]
\[ = - \frac{5}{48c^2} + \frac{ \log 2}{24 c^2} (1-c^2) + \frac{\gamma_e}{24c^2} (c^2-1)  - \frac 1 2 \int_1 ^\infty \frac{e^{ct}}{(e^t -1)(1-e^{ct})^2} dt \]
\[ - \frac 1 2 \int_0 ^1  \left[\frac{e^{c t}}{(e^t-1)(1-e^{ct})^2} - \left(  \frac{1}{c^2 t^3} - \frac{1}{2 c^2 t^2} + \frac{1}{12 t c^2} - \frac{1}{12 t} \right)\right] dt. \]
We compare this to the formulation in Proposition \ref{prop:xic}
\beq \frac{d}{dc} \left(\xi_c '(0)\right) &=& \frac 1 2 \int_1 ^\infty \frac 1 t \frac{1}{e^t - 1} \frac{-t e^{ct}}{(e^{ct} -1)^2} dt \nn\\
& &+ \frac 1 2  \int_0 ^1 \frac 1 t \left( \frac{-t e^{ct}}{(e^t -1)(e^{ct}-1)^2} + \frac{1}{c^2 t^2} - \frac{1}{2c^2 t} - \frac 1 {12} + \frac{1}{12 c^2} \right) dt  \nn\\
& & + \frac 1 2 \gamma_e \left( \frac{1}{12} - \frac{1}{12c^2} \right) - \frac{5}{48c^2} - \frac{1}{24} \log 2 \left( 1 - \frac 1 {c^2} \right) \nn \eeq
\[ = - \frac{5}{48c^2} + \frac{\log 2}{24c^2} (1-c^2) + \frac{\gamma_e}{24c^2}(c^2-1)  - \frac 1 2 \int_1 ^\infty  \frac{1}{e^t - 1} \frac{e^{ct}}{(e^{ct} -1)^2} dt \]
\[ - \frac 1 2  \int_0 ^1  \frac{ e^{ct}}{(e^t -1)(e^{ct}-1)^2} - \left( \frac{1}{c^2 t^3} - \frac{1}{2c^2 t^2} + \frac{1}{12 c^2t} -  \frac 1 {12t} \right) dt. \]
\end{proof}

\section{Concluding remarks}  \label{s:conclude}
We have calculated the variation of Barnes and Bessel zeta functions with respect to the essential parameter in their definitions.  It may be useful for computing further derivatives of higher dimensional zeta functions with potential applications in number theory, geometric analysis, and physics.  Moreover, our calculations revealed several identities involving both elementary and special functions.  These identities may be of independent interest or application.

\end{document}